\documentclass[reqno,12pt]{amsart}
\usepackage{a4wide}
\usepackage{amsmath}
\usepackage{amssymb}
\usepackage{amsthm}

\input texdraw

\numberwithin{equation}{section}

\newtheorem{theo}{Theorem}[section]
\newtheorem{conj}[theo]{Conjecture}
\newtheorem{coro}[theo]{Corollary}
\newtheorem{prop}[theo]{Proposition}
\newtheorem{lem}[theo]{Lemma}

\theoremstyle{remark}
\newtheorem*{Remark}{Remark}

\def\Punkt(#1 #2){\move(#1 #2)\fcir f:0 r:.10}
\def\DickPunkt(#1 #2){\move(#1 #2)\fcir f:0 r:.15}

\def\og{\leavevmode\raise0.25ex
                       \hbox{$\scriptscriptstyle\ll$}}
\def\fg{\leavevmode\raise0.25ex
                       \hbox{$\scriptscriptstyle\gg$}}
\ProvideTextCommandDefault{\guillemotleft}%
  {\leavevmode\hbox{\fontencoding{U}\fontfamily{lasy}%
                    \fontseries{m}\fontshape{n}\selectfont
    (\kern-0.20em(\kern+0.20em}\nobreak}
\ProvideTextCommandDefault{\guillemotright}%
  {\nobreak\leavevmode
   \hbox{\fontencoding{U}\fontfamily{lasy}
         \fontseries{m}\fontshape{n}\selectfont
   \kern+0.20em)\kern-0.20em)}}
\def\Log{\operatorname{Log}}
\def\ord{\operatorname{ord}}
\def\stab{\operatorname{stab}}

\def\al{\alpha}

\def\ga{\gamma}

\def\de{\delta}
\def\ep{\varepsilon}

\def\th{\theta}

\def\la{\lambda}

\def\si{\sigma}
\def\ta{\tau}

\def\ph{\varphi}

\def\om{\omega}
\def\Ga{\Gamma}
\def\De{\Delta}

\def\La{\Lambda}

\def\({\left(}
\def\){\right)}
\def\[{\left[}
\def\]{\right]}
\def\ii{\infty}

\def\lcm{\operatorname{lcm}}

\def\dd{\textup{d}}

\begin{document}

\title[]{Analytic properties of mirror maps}

\author[]{C. Krattenthaler$^\dagger$ and T. Rivoal$^\ddagger$}
\date{\today}

\address{C. Krattenthaler, Fakult\"at f\"ur Mathematik, Universit\"at Wien,
Nordbergstrasze~15, A-1090 Vienna, Austria.
WWW: \tt http://www.mat.univie.ac.at/\~{}kratt.}

\address{T. Rivoal,
Institut Fourier,
CNRS UMR 5582, Universit{\'e} Grenoble 1,
100 rue des Maths, BP~74,
38402 Saint-Martin d'H{\`e}res cedex,
France.\newline
WWW: \tt http://www-fourier.ujf-grenoble.fr/\~{}rivoal.}

\thanks{$^\dagger$Research partially supported 
by the Austrian Science Foundation FWF, grants Z130-N13 and S9607-N13,
the latter in the framework of the National Research Network
``Analytic Combinatorics and Probabilistic Number Theory''.
\newline\indent 
$^\ddagger$Research partially supported by the project 
Q-DIFF, ANR-10-JCJC-0105, of  
the french ``Agence Nationale de la Recherche''.}



\begin{abstract}
We consider a multi-parameter family of canonical coordinates and mirror maps originally introduced
by Zudilin [{\it Math.\ Notes} {\bf 71} (2002), 604--616]. This family includes
many of the known one-variable mirror maps as special cases, in particular many
of modular origin and the celebrated example of Candelas, de la Ossa, Green and Parkes
[{\it Nucl. Phys.} {\bf B359} (1991), 21--74] associated to the
quintic hypersurface in $\mathbb P^4(\mathbb C)$. 
In [{\it Duke Math.\ J.} {\bf 151} (2010), 175--218], we proved that all coefficients 
in the Taylor expansions at $0$ of these canonical coordinates (and, hence,
of the corresponding mirror maps) are integers.
Here we prove that all coefficients in the Taylor expansions at $0$ of
these canonical coordinates are positive. Furthermore, we provide
several results pertaining to the behaviour of the canonical
coordinates and mirror maps as complex functions. In particular, we
address analytic continuation, points of singularity, and radius of
convergence of these functions. We present several very precise 
conjectures on the radius of convergence of the mirror maps and the
sign pattern of the coefficients in their Taylor expansions at $0$.
\end{abstract}

\maketitle

\section{Introduction}

In the focus of this article there is 
a multi-parameter family of canonical coordinates and mirror maps originally introduced
by Zudilin \cite{zud} (to be defined below). This family contains
many of the known one-variable mirror maps as special cases, including many
of modular origin and the celebrated example of Candelas, de la Ossa, Green and Parkes
\cite{candelas} associated to the
quintic hypersurface in $\mathbb P^4(\mathbb C)$. 
For the {\it geometric} significance of these maps see \cite{candelas,morrisonjams,pandha,VoisAA}.
The {\it number-theoretic} properties of the coefficients in the Taylor
expansions at $0$ of these canonical coordinates and mirror maps have
been investigated recently in \cite{kratrivmirror} (cf.\ \cite{dela}
for a far-reaching generalisation).
Our aim here is to provide an as detailed as possible analysis of the
{\it analytic} properties of canonical coordinates and mirror
maps. Apart from the intrinsic interest in this kind of investigation,
one motivation comes from the hope of finding more applications of
the Diophantine method of {\guillemotleft\it uniformisation
ad\'elique simultan\'ee\guillemotright} of Andr\'e \cite{andre}, notably to
non-modular situations.

Another connection between number theory and mirror maps can be found
in the study of the arithmetic nature of values of the Riemann zeta
function at integers.
Ap\'ery's proof of the irrationality of $\zeta(3)$ (cf.\
\cite{vanPAA}) was recast in terms of modular forms by Beukers
\cite{BeukAA}. The search for an extension of Beukers' ideas to
$\zeta(n)$, $n\ge4$, led Almkvist and Zudilin \cite{almkvist} to study
systematically mirror maps associated to Fuchsian differential
equations, not necessarily of hypergeometric type. In \cite{KrRiAZ},
the authors showed that, in fact, many of the examples of Beukers and
of Almkvist and Zudilin can be obtained from suitable specialisations
of hypergeometric multi-variable mirror maps. It would
therefore be of great interest
to extend the investigations undertaken in this paper addressing
the analytic behaviour of the family of mirror maps
to be introduced below to the family of multi-variable mirror maps in
\cite{KrRiAZ}. 

\medskip
Let us now introduce this family of canonical coordinates and 
corresponding mirror maps.
For a given integer $N\ge 1$, let $r_1, r_2, \ldots, r_d$ denote 
the integers in $\{1, 2,\ldots, N\}$ which are coprime to $N$. It is
well-known that $d=\varphi(N)$, Euler's totient function, which is given by  
$
\varphi(N) = N 
{\prod_{p\mid N}} \big(1-\frac{1}{p}\big).
$
Set
$
C_N := N^{\varphi(N)}
{\prod_{p\mid N}} p^{\varphi(N)/(p-1)},
$
which is an integer because $p-1$ divides $\varphi(N)$ for any prime $p$
dividing $N.$ 
Let us also define the Pochhammer symbol $(\al)_m$ for complex numbers 
$\al$ and non-negative integers $m$ by
$(\al)_m:=\al(\al+1)\cdots (\al+m-1)$ if $m\ge1$, and $(\al)_0:=1$.
It can be proved (see~\cite[Lemma~1]{zud})  
that, for any integer $m\ge 0$,  
\begin{equation}\label{eq:definitionBNgras}
\mathbf{B}_N(m) := C_N^m \prod_{j=1}^{\varphi(N)} 
\frac{(r_j/N)_m}{m!}
\end{equation}
is an integer. 

Let us consider the hypergeometric differential operator ${\bf L}$ defined by 
\begin{equation}\label{eq:equadiff}
{\bf L}:=\left(z\frac{\dd}{\dd z}\right)^{\varphi(N_1)+\cdots +\varphi(N_k)} 
- C_{\bf N} z
\prod_{j=1}^k\prod_{i=1}^{\varphi(N_j)} 
\left(z\frac{\dd}{\dd z} + \frac{r_{i,j}}{N_j}\right).
\end{equation}
Here, $C_{\bf N}=C_{N_1}C_{N_2}\cdots C_{N_k}$ 
and the $r_{i,j}\in \{1, 2, \ldots, N_j\}$ form the residue 
classes modulo $N_j$ which are coprime to $N_j$. Unless $k=1$ and
$\mathbf N=(2)$, 
the differential equation ${\bf L}y=0$ is of order $\ge 2$.
We can construct two solutions
as follows. Set  
$H(x,m):= \sum_{n=0}^{m-1}\frac{1}{x+n}$ and 
$$
\mathbf{H}_N(m):= \sum_{j=1}^{\varphi(N)} H(r_j/N,m) - \varphi(N)H(1,m).
$$
Then, 
$\mathbf{F}_{\mathbf{N}}(z)$ and
$\mathbf{G}_{\mathbf{N}}(z)+\log(z)\mathbf{F}_{\mathbf{N}}(z)$ 
are two $\mathbb{C}$-linearly independent solutions to $\mathbf{L}y=0$, where
\begin{equation*} 
\mathbf{F}_{\mathbf{N}}(z) := \sum_{m=0}^{\infty} 
\bigg(
\prod_{j=1}^k
\mathbf{B}_{N_j}(m)
\bigg)
z^m
\end{equation*}
and
\begin{equation} \label{eq:GN} 
\mathbf{G}_{\mathbf{N}}(z) := \sum_{m=1}^{\infty} \bigg(\sum_{j=1}^k
\mathbf{H}_{N_j}(m) \bigg)\bigg( \prod_{j=1}^k \mathbf{B}_{N_j}(m)\bigg)
z^m,
\end{equation}
and where $\log(z)$ denotes the principal branch of the logarithm.
For simplicity, we write ${\bf B}_{\bf N}(m):=\prod_{j=1}^k
\mathbf{B}_{N_j}(m)$ and 
${\bf H}_{\bf N}(m):=\sum_{j=1}^k \mathbf{H}_{N_j}(m)$. 
Since $\mathbf B_1(m)=1$ and $\mathbf H_1(m)=0$ for all $m\ge0$, the
series $\mathbf{F}_{\mathbf{N}}(z)$ and $\mathbf{G}_{\mathbf{N}}(z)$
do not change if one omits or adds components of $1$ from/to $\mathbf
N$. We may therefore assume without loss of generality 
that $N_j\ge2$ for all $j$, which we shall
do throughout the paper.

The power series $\mathbf{F}_{\mathbf{N}}(z)$ and 
$\mathbf{G}_{\mathbf{N}}(z)$ have radius of convergence $1/C_{\mathbf
N}$.
We prove in Section~\ref{sec:analytic}
that the functions 
$\mathbf{F}_{\mathbf{N}}(z)$ and
$\mathbf{G}_{\mathbf{N}}(z)+\log(z)\mathbf{F}_{\mathbf{N}}(z)$ can be
analytically 
continued to $\mathbb{C}\setminus  [1/C_{\bf N}, +\infty)$ and 
$\mathbb{C}\setminus \big((-\infty, 0] \cup [1/C_{\bf N},
+\infty)\big)$, respectively.

Given the notation above, we define the {\it canonical coordinate}
${\bf q}_{\mathbf{N}} (z)$ as the exponential of the quotient of the 
above two solutions, that is, by
\begin{equation} \label{eq:canon} 
{\bf q}_{\mathbf{N}} (z):= z 
\exp(\mathbf{G}_{\mathbf{N}}(z)/\mathbf{F}_{\mathbf{N}}(z)).
\end{equation}
Its compositional inverse, which we denote by ${\bf z}_{\mathbf{N}}
(q)$, is called (the corresponding) {\it mirror map}.

When $k=1$ and $\mathbf N=(2)$, we have ${\bf F}_{(2)}(z)=(1-4z)^{-1/2}$, 
which satisfies the differential equation 
$(1-4z)y'-2y=0.$ 
The function $\mathbf{G}_{(2)}(z)+\log(z)\mathbf{F}_{(2)}(z)$ 
defined formally by the above 
formula is not solution of that differential equation,  
but it turns out that all theorems stated below are still true in this case 
because 
\begin{equation} \label{eq:Phi=1} 
{\bf q}_{(2)}(z)= 
(1-\sqrt{1-4z})^2/(4z).
\end{equation}
However, certain 
 proofs do not work for this case, and we will say when. 

The special case $\mathbf N=(5)$ has been of particular interest
since it produces the
earlier mentioned example of Candelas et al.\ \cite{candelas}.

\medskip
In~\cite{kratrivmirror}, we proved that,
for any positive integers $N_1, N_2, \dots, N_k$, the canonical coordinate
${\bf q}_{\mathbf{N}} (z)$ has integral Taylor coefficients.
Our first result says that these coefficients are, in fact, positive.
Its proof is given in Section~\ref{sec:pos}. An essential ingredient there
is a classical result of Kaluza \cite{kaluza} on the sign of
coefficients in certain power series expansions (see
Lemma~\ref{lem:2a}).

\begin{theo}\label{theo:1} For all integers $N_1,N_2, \ldots, N_k\ge 2$,
all Taylor coefficients of ${\bf q}_{\mathbf{N}} (z)$ at $0$ are
positive, except the constant coefficient.
\end{theo}

A problem that suggests itself at this point is to find a
combinatorial interpretation for the Taylor coefficients of ${\bf
q}_{\mathbf{N}} (z)$ or of ${\bf
z}_{\mathbf{N}} (q)$ (even if the latter may have negative
coefficients, see Conjecture~\ref{conj:2} below). 
Some progress in this direction can be found in \cite{LaLWAA}.

The next theorem provides precise information on the radius of
convergence and the asymptotic behaviour of the Taylor coefficients of
the canonical coordinate ${\bf
q}_{\mathbf{N}} (z)$ as a power series in $z$.
Here, and in the sequel, given $\mathbf N=(N_1,N_2,\dots,N_k)$, we employ the
notation 
\begin{equation} \label{eq:Phi}
\Phi_{\mathbf{N}}:=\sum _{j=1} ^{k}\varphi(N_j). 
\end{equation}

\begin{theo} \label{theo:2} For all integers $N_1,N_2, \ldots, N_k\ge2$,
the following assertions hold:

\begin{enumerate}
\item[$(i)$]  
The radius of convergence of the Taylor series of ${\bf q}_{\mathbf{N}} (z)$ is 
equal to $1/C_{\bf N}$ and the Taylor series converges for any $z$ such 
that $\vert z\vert=1/C_{\bf N}$.

\item[$(ii)$]  
The function ${\bf q}_{\mathbf{N}} (z)$ has a singularity at 
$z=1/C_{\mathbf{N}}$.

\item[$(iii)$]  
For any $z$ such that $\vert z\vert \le 1/C_{\bf N}$, we have
$\vert {\bf q}_{\mathbf{N}}(z)\vert \le 1.$
\item[$(iv)$]  
If $\Phi_{\mathbf{N}}=1$, then
the $m$-th Taylor coefficient of ${\bf q}_{(2)}(z)$ is equal to the
$m$-th Catalan number $\frac {1} {m+1}\binom {2m}m$, $m\ge1$, and,
hence, as $m\to\infty$, it is equal to
$$
\frac {4^m} {\sqrt\pi m^{3/2}}\left(1+o(1)\right).
$$
\item[$(v)$]  
If $\Phi_{\mathbf{N}}=2$, then, as $m$ tends to $\infty$, 
the $m$-th Taylor coefficient of ${\bf q}_{\bf  N}(z)$ is equal to
$$
const.\times \frac {C_{\mathbf N}^m} {m\log^2(m)}\left(1+o(1)\right),
$$
where throughout the symbol ``$const.$" stands for a non-zero constant.
\item[$(vi)$]  
If $\Phi_{\mathbf{N}}\ge3$, then, as $m$ tends to $\infty$, 
the $m$-th Taylor coefficient of ${\bf q}_{\bf  N}(z)$ is equal to
$$
const.\times \frac {C_{\mathbf N}^m} {m^{\Phi_{\mathbf{N}}/2}}\left(1+o(1)\right).
$$
\end{enumerate}
\end{theo} 

In $(iii)$, the inequality is always strict, except at $z=1/C_{\bf N}$
when the series ${\bf F}_{\bf N}(1/C_{\bf N})$ diverges to $+\infty$, which 
happens in five cases: $k=1, N_1=2,3,4$ or $6$, and $k=2$, $N_1=N_2=2$. 

Theorem~\ref{theo:2} is proved in Section~\ref{sec:theo2}. 
There, the analytic
continuation of ${\bf F}_{\bf N}(z)$ and ${\bf G}_{\bf N}(z)$, which
is discussed in Section~\ref{sec:analytic}, plays an important role,
as well as the fine behaviour of these functions near the point
$1/C_{\mathbf N}$, which is discussed in detail in
Section~\ref{sec:aux}.

\medskip
By \eqref{eq:canon}, the canonical coordinate ${\bf q}_{\mathbf{N}}
(z)$ is only defined for $z$ in the disk of convergence of the series 
$\mathbf{G}_{\mathbf{N}}(z)$ and $\mathbf{F}_{\mathbf{N}}(z)$ involved
in its definition. The knowledge of the analytic
continuation of ${\bf F}_{\bf N}(z)$ and ${\bf G}_{\bf N}(z)$ from 
Section~\ref{sec:analytic}, combined with a theorem of P\'olya \cite{polya} on
zeroes of hypergeometric functions, allows us to show that ${\bf q}_{\mathbf{N}}
(z)$ can be continued to a function of the entire complex plane except
for a branch cut. The corresponding proof is the subject of
Section~\ref{sec:polya}. 

\begin{theo}\label{theo:3} For all integers $N_1,N_2, \ldots, N_k\ge2$,
the following assertions hold:

\begin{enumerate}
\item[$(i)$]  
The 
power series ${\bf q}_{\mathbf{N}} (z)$ can be continued to an analytic function on 
$\mathbb{C}\setminus [1/C_{\bf N},+\infty)$. 
\item[$(ii)$]  
The point $1/C_{\bf N}$ is a branch point.
\item[$(iii)$]  
We have
\begin{equation}\label{eq:liminfini}
\lim_{z \to \infty} {\bf q}_{\mathbf{N}}(z) =
-\exp\big(-\pi\cot(\pi/M_{\mathbf N})\big),
\end{equation}
where $M_{\mathbf N}=\max(N_1, \ldots, N_k)$, and where the limit has
to be performed along a path that avoids the cut $[1/C_{\mathbf N},+\infty)$.
\end{enumerate}
\end{theo}

The monodromy of ${\bf q}_{\mathbf{N}}(z)$ at $z=1/C_{\mathbf N}$
follows from applying the well-known monodromy theory of solutions to
(generalised) hypergeometric equations (cf.\ \cite{BeHeAA}) to the series
$\mathbf{F}_{\mathbf{N}}(z)$ and 
$\mathbf{G}_{\mathbf{N}}(z)+\log(z)\mathbf{F}_{\mathbf{N}}(z)$.

\medskip
The function ${\bf q}_{\bf N}(z)$, seen as a formal power series in $z$,
is invertible at $z=0$. Its formal inverse  ${\bf z}_{\bf N}(q)$ is 
a power series in $q$ which converges in a neighbourhood of
$q=0$ (cf.\ \cite[Theorems~2.4b, 2.4c]{HenrAA}). 
The corresponding results for this compositional 
inverse ${\bf z}_{\mathbf{N}}
(q)$ that we are able to establish are less complete than
Theorems~\ref{theo:1}--\ref{theo:3} for ${\bf q}_{\mathbf{N}}(z)$. 
First of all, we use
Theorem~\ref{theo:3} to get some information on the radius of
convergence of ${\bf z}_{\bf N}(q)$ as a power series in $q$.

\begin{theo} \label{theo:4}
Let $N_1,N_2, \ldots, N_k$ be integers, all at least $2$. 

\begin{enumerate}
\item[$(i)$]
The radius 
of convergence of the Taylor 
series of  ${\bf z}_{\bf N}(q)$ is at most\break 
$\exp\big(-\pi\cot(\pi/M_{\mathbf N})\big)$.
\item[$(ii)$]
If $\Phi_{\mathbf{N}}\ge4$, then the radius of convergence of the Taylor 
series of ${\bf z}_{\bf N}(q)$ is at most ${\bf q}_{\bf N}(1/C_{\bf N})$.
\end{enumerate}
\end{theo}

The proof of this theorem is found in Section~\ref{sec:rad}.

Lemma~\ref{lem:exp>q(1/C)} in Section~\ref{sec:comp} provides a comparison of the two values
that feature in the above theorem in the case where $\Phi_{\mathbf
N}\ge4$: namely, there it is shown that
$\exp(-\pi\cot(\pi/M_{\mathbf N}))>{\bf q}_{\bf N}(1/C_{\bf N})$.
As we shall see in Section~\ref{sec:Phi}, this inequality continues to
hold for $\Phi_{\mathbf N}=3$, but it is wrong for the remaining cases 
$\Phi_{\mathbf N}=1,2$.

\medskip
In the cases of ``small" $\Phi_{\mathbf N}$, we are able to provide
precise information on the analytic continuation of ${\bf z}_{\bf
N}(q)$, see the theorem below. 
As we explain in Section~\ref{sec:Phi}, this information is
in fact essentially
available in the literature on the modular origin of the corresponding
mirror maps.
We point out that
there exist several other hypergeometric functions (not covered by our
series ${\bf F}_{\bf N}(z)$) which give rise to mirror maps of modular
origin, see e.g.\ \cite{StillAA,zudram}. It would be of interest to know
if there is a result analogous to the one below for these mirror maps.

\begin{theo} \label{theo:5}
For all integers $N_1,N_2, \ldots, N_k\ge2$ with $\Phi_{\mathbf{N}}\le 3$,
the radius 
of convergence of the Taylor 
series of  ${\bf z}_{\bf N}(q)$ is equal to
$\exp\big(-\pi\cot(\pi/M_{\mathbf N})\big)$. The function ${\bf z}_{\bf N}(q)$ can
be analytically continued to the unit disk with the exception of a set
of poles or branch points which can be described precisely. In
particular, $-\exp\big(-\pi\cot(\pi/M_{\mathbf N})\big)$ is always a
singularity of ${\bf z}_{\bf N}(q)$. If
$\mathbf N=(2)$, we have ${\bf z}_{\bf N}(q)=q/(1+q)^2$.
In all the eight other cases, the unit
circle forms a natural boundary for ${\bf z}_{\bf N}(q)$.
\end{theo}

In Section~\ref{sec:Phi}, one
also finds precise information on the poles, respectively branch points.

In the case where $\Phi_{\mathbf N}\ge4$, we can offer only the
partial result below concerning the analytic nature of ${\bf z}_{\bf
N}(q)$, with the proof given in Section~\ref{sec:prop}. 
There, and also later, we use the notion of
a {\it right  slit neighbourhood\/} of $z_0$, by which
we mean a domain of the form (see Figure~\ref{fig:1})
\begin{equation} \label{eq:slit} 
\De(z_0;r,\ep):=
\{z:\vert z-z_0\vert<r\text{ and }\vert\arg(z-z_0)\vert> \ep\}
\end{equation}
for some $r,\ep>0$, where it is understood that $\arg(z-z_0)$ is taken
from $(-\pi,\pi]$. A left slit neighbourhood is defined in the obvious 
analogous way.

\begin{figure}
\centertexdraw{
  \drawdim truecm  \linewd.05
\move(0 0)
  \rlvec(2 .5)
\move(0 0)
  \rlvec(2 -.5)
\move(0 0)
\larc r:2.1 sd:14 ed:346
\Punkt(0 0) 
\htext(-.6 -0.2){$z_0$}

\linewd.01
\move(0 0)
  \rlvec(3 0)
\move(0 0)
\larc r:2.1 sd:0 ed:14
\htext(1.8 .15){$\ep$}
\move(0 0)
  \ravec(0 2.1)
\htext(.1 1){$r$}
\move(-3 0)
}
\vskip40pt
\caption{The right slit neighbourhood $\De(z_0;r,\ep)$ of $z_0$}
\label{fig:1}
\end{figure}

\begin{prop} \label{prop:1}
For all integers $N_1,N_2, \ldots, N_k\ge2$ with
  $\Phi_{\mathbf{N}}\ge 4$,
the series ${\bf z}_{\bf N}(q)$ can be analytically continued to a 
domain that contains the half-open segment $[0,{\bf q}_{\bf
  N}(1/C_{\mathbf N}))$.
Moreover, this domain can be chosen so that it contains as
well a right slit neighbourhood  
of ${\bf q}_{\bf
  N}(1/C_{\mathbf N})$ {\em(}where the angle of the slit, controlled by the
parameter $\ep$ in \eqref{eq:slit}, can be chosen arbitrarily small\/{\em)}.
\end{prop}

\begin{Remark}
The function ${\bf q}_{\bf N}(z)$ can also be locally inverted at
$z=\infty$. Consequently, the\break corresponding local 
inverse function $\tilde{\bf z}_{\bf N}(q)$ is defined in
a left slit neighbour-\break hood of $-\exp(-\pi\cot(\pi/M_{\mathbf N}))$
(see the last paragraph in Section~\ref{sec:prop}).
The point\break
$-\exp(-\pi\cot(\pi/M_{\mathbf N}))$ is always a singular point.
Whether it is a pole, a branch point, or an essential singularity,
this can be decided in each case
by inspecting the expansions at $\infty$ for ${\bf q}_{\bf
  N}(z)$ given in Lemma~\ref{lem:expinf} (taking also into
consideration the remark after the statement of the lemma), and by
subsequently applying Lemma~\ref{lem:qz}, with the accompanying remark
in mind. Unfortunately, we do not know how to relate
$\tilde{\bf z}_{\bf N}(q)$ and ${\bf z}_{\bf N}(q)$.
\end{Remark}

Certainly, the above result cannot even answer the innocent question of
what the radius of convergence of ${\bf z}_{\bf  N}(q)$ as a power
series in $q$ is. However, we performed extensive calculations using
the computer algebra system PARI/GP \cite{PARI}, in which we computed
approximations to the radius of convergence using the quotient rule up
to many digits. These computations provide abundant evidence for the
exact value of the radius of convergence for the cases where
$\Phi_{\mathbf N}$ is ``large."

\begin{conj} \label{conj:1}
For all integers $N_1,N_2, \ldots, N_k\ge2$ with $\Phi_{\mathbf{N}}\ge4$,
the radius of convergence of the Taylor 
series of  ${\bf z}_{\bf N}(q)$ is equal to ${\bf q}_{\bf N}(1/C_{\bf N})$.
Moreover, the point ${\bf q}_{\bf N}(1/C_{\bf N})$ is the only
singularity of ${\bf z}_{\bf N}(q)$ on the boundary of its disk of convergence.
\end{conj}

We discuss a possible line of argument to prove this conjecture
in Section~\ref{sec:What}.

\medskip

While, for $\Phi_{\mathbf N}=2,3$, Theorem~\ref{theo:5} shows that the
mirror map ${\bf z}_{\bf N}(q)$ has a natural boundary, our computer
experiments for $\Phi_{\mathbf N}\ge 4$ did not allow us to develop an
intuition whether or not there is a natural boundary for ${\bf z}_{\bf
N}(q)$ in these cases.

Concerning the Taylor coefficients of ${\bf z}_{\bf N}(q)$ at $q=0$,
our numerical calculations suggest a very predictable behaviour.

\begin{conj} \label{conj:2}
Let $N_1,N_2, \ldots, N_k$ be positive integers, all at least $2$.

\begin{enumerate}
\item[$(i)$]
If $\Phi_{\mathbf{N}}\le3$, the Taylor coefficients of 
${\bf z}_{\bf N}(q)$ have alternating signs.
\item[$(ii)$]
If $\Phi_{\mathbf{N}}=4$, the coefficients of $q$ and of $q^3$ in
the Taylor series of ${\bf z}_{\bf N}(q)$ are positive, while all
other coefficients are negative, except the constant coefficient,
which vanishes.  
\item[$(iii)$]
If $\Phi_{\mathbf{N}}\ge5$, the coefficient of $q$ in
the Taylor series of ${\bf z}_{\bf N}(q)$ is positive, while all
other coefficients are negative, except the constant coefficient,
which vanishes.  
\end{enumerate}
\end{conj}

We show in Section~\ref{sec:conj} that Conjecture~\ref{conj:2} implies
the first assertion in Conjecture~\ref{conj:1}. 
On the other hand, as we show in Section~\ref{sec:conj1weak},
Conjecture~\ref{conj:1} implies a weak version of
Conjecture~\ref{conj:2}.

There are three cases, where
Conjecture~\ref{conj:2}.$(i)$ is a theorem:
in the case $\mathbf N=(2,6)$, where 
${\bf z}_{(2,6)}(q)=1/j(\tau)$ (with $q=\exp(2i\pi\tau)$ and $j(\tau)$
the Dedekind--Klein $j$-invariant; see Section~\ref{sec:Phi=3}), this has been
established by Asai, Kaneko and Ninomiya
\cite[Cor.~2, p.~94]{AsKaAA}. 
In the case $\mathbf N=(2,2)$, alternance of coefficients follows immediately
upon inspection of the explicit expression for ${\bf z}_{(2,2)}(q)$ 
given in \eqref{eq:z22}.
Similarly, 
in the case $\mathbf N=(2,2,2)$, alternance of coefficients follows immediately
upon inspection of the explicit expression for ${\bf z}_{(2,2,2)}(q)$ 
given in \eqref{eq:z222}.
Moreover, Theorem~\ref{theo:5}, Lemma~\ref{lem:exp1/C}, 
Lemma~\ref{lem:qz}, together with 
the standard theorems of singularity analysis (see 
\cite[Ch.~VI]{FlSeAA}) imply that Conjecture~\ref{conj:2}.$(i)$
holds ``asymptotically," meaning that the coefficients of $q^m$ of 
${\bf z}_{\mathbf N}(q)$ alternate in sign for all sufficiently large
$m$. 

The expression for ${\bf q}_{\bf N}(z)$ in terms of ${\bf F}_{\bf
N}(z)$ and ${\bf G}_{\bf N}(z)$ is not very convenient for a study
of fine analytic properties of ${\bf q}_{\bf N}(z)$ and of
the corresponding mirror map ${\bf z}_{\bf N}(q)$. The following
conjecture, extending Theorem~\ref{theo:1}, 
is motivated by the search for an alternative expression,
as we explain below.

\begin{conj} \label{conj:5}
For all integers $N_1,N_2, \ldots, N_k\ge2$ and positive integers $n$, 
the coefficient of $z^{m}$ in $(C_{\mathbf N}z-1)^n{\bf q}_{\bf N}(z)$ is
positive for all $m\ge n+1$. 
\end{conj}

Computer calculations indicate that, actually, the coefficients of
$z^m$ in  $(C_{\mathbf N}z-1)^n{\bf q}_{\bf N}(z)$ are alternating for
$m\le n+1$ until the sign stabilises as described in the conjecture. 
That it must stabilise eventually follows from the singular expansion
for  ${\bf q}_{\bf N}(z)$ given in Lemma~\ref{lem:exp1/C} 
together with 
the standard theorems of singularity analysis (see 
\cite[Ch.~VI]{FlSeAA}).
So the point here is that the sign stabilises {\it already} for
$m\ge n+1$.

By a classical theorem of Hausdorff \cite{HausAA},
Conjecture~\ref{conj:5} implies that the sequence of coefficients of 
${\bf q}_{\bf N}(z)$ is a moment sequence for a finite measure
$d\chi(t)$ on
$(0,C_{\mathbf N})$. As a consequence, ${\bf q}_{\bf N}(z)$ could be
written in the form
$$
{\bf q}_{\bf N}(z)=z\int _{0} ^{C_{\mathbf N}}\frac {\dd\chi(t)}
{1-zt}.
$$
By a theorem of Wirths \cite{WirtAA}, it would then follow that
${\bf q}_{\bf N}(z)$ is univalent in the open half plane
$\text{Re}(z)<1/C_{\mathbf N}$. In particular, 
this would lead to the following strengthening of
Proposition~\ref{prop:1}: 
For all integers $N_1,N_2, \ldots, N_k\ge2$ with
  $\Phi_{\mathbf{N}}\ge 4$,
the series ${\bf z}_{\bf N}(q)$ can be analytically continued to a 
domain that contains the open segment 
$(-\exp(-\pi\cot(\pi/M_{\mathbf N})),{\bf q}_{\bf
  N}(1/C_{\mathbf N}))$.
Moreover, this domain can be chosen so that it contains as
well a right slit neighbourhood of ${\bf q}_{\bf
  N}(1/C_{\mathbf N})$ and a left slit neighbourhood of
$-\exp(-\pi\cot(\pi/M_{\mathbf N}))$ {(}where the angle of the slit
can be chosen arbitrarily small; cf.\ the last paragraph in
Section~\ref{sec:prop}{)}.

\medskip
From now on, in order to simplify notation, we let $C:=C_{\bf N}$,
$\Phi:=\Phi_{\bf N}$, and $M:=M_{\bf N}$.

\section{Proof of Theorem~\ref{theo:1}}
\label{sec:pos}

A formal power series $f(z)=\sum_{n=0}^{\infty} f_n z^n\in\mathbb{C}[[z]]$ with 
$f_0\neq 0$ is invertible 
in $\mathbb{C}[[z]]$. We can thus see  $\widehat{f}(z)=1-1/f(z)$ 
as a formal power series. 
For the proof of Theorem~\ref{theo:1}, we need the following auxiliary
result, whose proof can be found further down in this section.

\begin{lem} \label{lem:1} Let $f(z)=\sum_{n=0}^{\infty} f_nz^n$, $f_0=1$, 
be such that the Taylor coefficients 
of $\widehat{f}(z)$ are non-negative. Let us consider
$g(z)=\sum_{n=0}^{\infty} h_nf_nz^n$ where the sequence of real numbers
$(h_n)_{n \ge 0}$ is non-decreasing and non-negative. 
Then the  Taylor coefficients of $g(z)/f(z)$ are non-negative.

If, in addition, all Taylor coefficients of $f$ and $\widehat f$ are
positive {\em(}except the constant coefficient of $\widehat f${\em)} 
and the sequence $(h_n)_{n \ge 0}$ is strictly increasing,
then the  Taylor coefficients of $g(z)/f(z)$ are positive, except
the constant coefficient if $h_0=0$.
\end{lem}

To apply Lemma~\ref{lem:1} to our situation, we need two further results.
The first is a theorem due to Kaluza \cite[Satz~3]{kaluza}. Strictly
speaking, the assertion in Satz~3 in \cite{kaluza} deals only with the
condition $f_{n+1}f_{n-1}\ge f_n^2$; however, the strengthening
given below is easily extracted from the proof in \cite{kaluza}. 

\begin{lem} \label{lem:2a}
Let $f(z)=\sum_{n=0}^{\infty} f_nz^n$, $f_0=1$, 
be such that $f_1>0$ and $f_{n+1}f_{n-1}\ge f_n^2$ for all $n\ge1$. Then 
the Taylor coefficients 
of $\widehat{f}(z)$ are non-negative.

If the stronger condition $f_{n+1}f_{n-1}> f_n^2$ is satisfied for all
$n\ge1$, then the Taylor coefficients 
of $\widehat{f}(z)$ are positive {\em(}except the constant coefficient{\em)}.
\end{lem}

(It is easy to see that, if $\widehat{f}(z)$ has non-negative 
Taylor coefficients, then $f_n\ge 0$ for all $n\ge0$, but the converse 
is not true in general.) 

\begin{lem}\label{lem:2} 
Let us fix the 
integers $N_1,N_2, \ldots, N_k\ge 2$. 
We 
have  $\mathbf{B}_{\mathbf{N}}(1)>0$ and $$
\mathbf{B}_{\mathbf{N}}(m+1)
\mathbf{B}_{\mathbf{N}}(m-1)> \mathbf{B}_{\mathbf{N}}(m)^2
$$ 
for all $m\ge 1$. Furthermore, the sequence 
$({\bf H}_{\bf N}(m))_{m\ge 0}$ is positive and 
strictly increasing. 
\end{lem}

The proof of this lemma can be found at the end of this section.
The combination of Lemmas~\ref{lem:2a} and \ref{lem:2} immediately implies 
the following corollary.

\begin{coro} \label{cor:1}
Let $N_1,N_2, \ldots, N_k$ be positive integers, all at least $2$.
Then the Taylor coefficients of $\widehat{\bf F}_{\mathbf{N}}(z)$ are 
positive {\em(}except the constant coefficient{\em)}. 
\end{coro}

The above corollary shows that we can take 
$f={\bf F}_{\mathbf{N}}$ and $g={\bf G}_{\mathbf{N}}$ 
in Lemma~\ref{lem:1}.  Hence, the power series 
${\bf G}_{\mathbf{N}}(z)/{\bf F}_{\mathbf{N}}(z)$ 
has positive Taylor 
coefficients (except the constant coefficient), 
a property which obviously remains true when we take the 
exponential. This proves Theorem~\ref{theo:1}.

\begin{proof}[Proof of Lemma~\ref{lem:1}]
Let us write 
$
\widehat{f}(z)=\sum_{n=1}^{\infty} \widehat{f}_n z^n.
$
The relation $f(z)\big(1-\widehat{f}(z)\big)=1$ translates into
\begin{equation}\label{eq:fn}
f_n-\sum_{k=1}^{n} \widehat{f}_{k}f_{n-k}=\delta_{n,0},
\quad \text {for all $n\ge 0$},
\end{equation}
where $\delta_{n,0}$ is the Kronecker symbol. 
Furthermore, we have
$$
\frac{g(z)}{f(z)}=g(z)\big(1-\widehat{f}(z)\big)= \sum_{n=0}^{\infty} 
z^n\big(h_nf_n-\sum_{k=1}^{n} \widehat{f}_kf_{n-k}h_{n-k} \big).
$$

Since $(h_n)_{n \ge 0}$ is non-decreasing and non-negative, 
by~\eqref{eq:fn} we have
\begin{equation} \label{eq:ineq} 
h_nf_n-\sum_{k=1}^{n} \widehat{f}_kf_{n-k}h_{n-k} 
\ge h_nf_n-\sum_{k=1}^{n} \widehat{f}_kf_{n-k}h_{n} 
=h_n\Big(f_n-\sum_{k=1}^{n} \widehat{f}_kf_{n-k} \Big) =h_n\delta_{n,0}\ge 0.
\end{equation}

The additional assertion in the lemma follows as well from the above
arguments by observing that, because of the stronger assumptions, the
first inequality in \eqref{eq:ineq} is strict for $n\ge1$.
\end{proof}

\begin{proof}[Proof of Lemma~\ref{lem:2}]
It is clear that $\mathbf{B}_{\mathbf{N}}(m)>0$ for all $m \ge 0$.

We want to prove that
\begin{equation}\label{eq:B1}
\frac{\mathbf{B}_{\mathbf{N}}(m+1)
\mathbf{B}_{\mathbf{N}}(m-1)}
{\mathbf{B}_{\mathbf{N}}(m)^2}> 1
\end{equation}
for all $m\ge 1$. By definition, we have 
$$
\frac{\mathbf{B}_{\mathbf{N}}(m+1)
\mathbf{B}_{\mathbf{N}}(m-1)}
{\mathbf{B}_{\mathbf{N}}(m)^2} =
\prod_{j=1}^k\prod_{i=1}^{\varphi(N_j)} 
\frac{(r_{i,j}/N)_{m+1}(r_{i,j}/N)_{m-1}}{(r_{i,j}/N)_{m}^2}.
$$ 
We observe that for every real number $x>0$ and every integer $m\ge 1$, 
we have
$$
\frac{(x)_{m+1}(x)_{m-1}}{(x)_{m}^2} = \frac{x+m}{x+m-1}>1,
$$
which immediately implies~\eqref{eq:B1}.

\medskip

Concerning the second claim, we have 
\begin{align*}
\mathbf{H}_{\bf N}(m) &= \sum_{j=1}^{k}\bigg(
\sum_{i=1}^{\varphi(N_j)} H(r_{i,j}/N,m) - \varphi(N_j)H(1,m)\bigg)
\\
& =\sum_{j=1}^{k} \sum_{i=1}^{\varphi(N_j)} \Big(H(r_{i,j}/N_j,m) -H(1,m)\Big),
\end{align*}
from which we deduce that
$$
\mathbf{H}_{\bf N}(m+1)-\mathbf{H}_{\bf N}(m) = \sum_{j=1}^{k} 
\sum_{i=1}^{\varphi(N_j)} \left(\frac{1}{m+r_{i,j}/N_j}-\frac{1}{m+1}\right)> 0,
$$
because $0<r_{i,j}/N_j <1$. Since $\mathbf{H}_{\bf N}(0)=0$, we have 
proved that the 
sequence $(\mathbf{H}_{\bf N}(m))_{m\ge1}$ is positive and strictly increasing.
\end{proof}

\section{Analytic continuation of ${\bf F}_{\bf N}(z)$ and ${\bf G}_{\bf N}(z)$}
\label{sec:analytic}

For $z$ complex with $\vert \arg(-z)\vert <\pi$, let $\Log(z)$ denote
the branch of the logarithm which assigns values with imaginary part
between $0$ and $2\pi$. We shall continue to use $\log(\,.\,)$ for the
principal branch of the logarithm. In order to facilitate the reading
of the following paragraphs, as rule of thumb, below, whenever there
appears $z^\ga$, it has to be understood as $z^\ga=\exp(\ga\Log(z))$,
whereas whenever there
appears $(-z)^\ga$, it has to be understood as
$(-z)^\ga=\exp(\ga\log(-z))$.

For a real number $h \ge 0$, set 
\begin{equation} \label{eq:hyper2} 
F(h,z):=\sum_{n=0}^{\infty} \frac{(\al_1+h)_n\cdots (\al_\Phi+h)_n}{(1+h)_n^\Phi} z^{n+h},
\end{equation}
where $\Phi=\Phi_{\bf N}$ is given by \eqref{eq:Phi} 
and the $\al$'s run through the elements of 
the multiset~{(}\footnote{A multiset 
is a ``set'' where one allows repetitions of elements.}{)}  
$$\{r_{i,j}/N_j : i=1, \ldots, \varphi(N_j), j=1,\ldots,k\}.$$
We have 
\begin{equation}\label{eq:FG}
F(0,Cz)={\bf F}_{\bf N}(z) \quad \textup{and} \quad  
\frac{\partial F}{\partial h} (0,Cz) = \Log(Cz){\bf F}_{\bf N}(z) +{\bf G}_{\bf N}(z). 
\end{equation}
Generalised hypergeometric functions, such as the sum on the
right-hand side of \eqref{eq:hyper2} have a Barnes-type integral representation,
see \cite[Sec.~4.6]{SlatAC}. If we apply this to the right-hand side of
\eqref{eq:hyper2} then, for 
any complex number $z$ such that 
$\vert \arg(-z)\vert <\pi$ and any $h\ge 0$, we obtain
\begin{multline} \label{eq:sumint}
\sum_{n=0}^{\infty} \frac{(\al_1+h)_n\cdots (\al_\Phi+h)_n}{(1+h)_n^\Phi} z^{n}=
\\
-\frac{1}{2i\pi}\frac{\Gamma(1+h)^\Phi}{\Gamma(\al_1+h)\cdots \Gamma(\al_\Phi+h)}
\int_\mathcal{C} \frac{\Gamma(\al_1+h+s)\cdots \Gamma(\al_\Phi+h+s)}{\Gamma(1+h+s)^\Phi}
 \frac{\pi}{\sin(\pi s)} 
(-z)^{s} \dd s,
\end{multline}
where $\mathcal{C}$ is a path from $-i\infty$ to $+ i \infty$ such that 
$0,1, 2, \ldots$ lie on the right of 
$\mathcal{C}$ and the poles of the $\Gamma(\al_\ell+h+s)$, $\ell=1,
\ldots, \Phi$, lie to the left. 

By multiplying both sides of \eqref{eq:sumint} by
$z^h$, and by using the relation $\Log(z)=\log(-z)+i\pi$ (recall the
convention on the branches of the logarithm that we made in the first
paragraph of this section),
one obtains for $\vert \arg(-z)\vert <\pi$ and $h\ge 0$ the equation
\begin{multline}\label{eq:barnes}
F(h,z)=
\\
-\frac{e^{i\pi h}}{2i\pi}\frac{\Gamma(1+h)^\Phi}{\Gamma(\al_1+h)\cdots \Gamma(\al_\Phi+h)}
\int_\mathcal{C} \frac{\Gamma(\al_1+h+s)\cdots \Gamma(\al_\Phi+h+s)}{\Gamma(1+h+s)^\Phi}
 \frac{\pi}{\sin(\pi s)} 
(-z)^{s+h} \dd s.
\end{multline}
In particular, for $h=0$ and $z$ changed to $Cz$, this provides the analytic continuation 
of ${\bf F}_{\bf N}(z)$ to the cut plane $\vert \arg(-z)\vert <\pi$. Since ${\bf F}_{\bf N}(z)$ 
is also analytic at any 
real point $z \in [0,1/C[$, we get in this way the analytic continuation 
of ${\bf F}_{\bf N}(z)$ to the cut plane $\vert \arg(1/C-z)\vert <\pi$.

We now differentiate both sides of~\eqref{eq:barnes} with respect to $h$, and then set $h=0$ 
and change $z$ to $Cz$. After simplification, 
we get 
\begin{multline}\label{eq:barnes2}
{\bf G}_{\bf N}(z) = \Big(\Phi \psi(1)-\sum_{\ell=1}^\Phi \psi(\al_\ell)\Big){\bf F}_{\bf N}(z)
\\
-\frac{1}{2i\pi\Gamma(\al_1)\cdots \Gamma(\al_\Phi)}
\int_\mathcal{C} \frac{\Gamma(\al_1+s)\cdots \Gamma(\al_\Phi+s)}{\Gamma(1+s)^\Phi} 
\Big(\sum_{\ell=1}^\Phi \psi(\al_\ell+s) - \Phi\psi(1+s)\Big)\frac{\pi(-Cz)^{s}}{\sin(\pi s)} 
\dd s,
\end{multline}
where $\psi$ is the digamma function.
A standard argument shows that the integral on the right-hand side of~\eqref{eq:barnes2} 
is analytic in the cut plane $\vert \arg(-z)\vert <\pi$, hence this is
also the case for ${\bf G}_{\bf N}(z)$.
Then, from the series representation \eqref{eq:GN}, we conclude that, in fact, 
${\bf G}_{\bf N}(z)$ is analytic in the cut 
plane $\vert \arg(1/C-z)\vert <\pi$.

In~\cite[p.~216]{kratrivmirror}, we proved that 
\begin{equation}\label{eq:logC}
\Big(\Phi \psi(1)-\sum_{\ell=1}^\Phi \psi(\al_\ell)\Big) = 
\sum_{j=1}^k \sum_{i=1}^{\varphi(N_j)}\big(\psi(1)-\psi(r_{i,j}/N)\big)
=\log(C).
\end{equation}
This quantity will reappear in the sequel.

We now prove the following result, which will be needed in the proof
of Theorem~\ref{theo:3}.$(iii)$ in Section~\ref{sec:polya}.

\begin{lem}\label{eq:asympG/F}
For integers $N_1,\ldots, N_k\ge 2$, let $M=\max(N_1, \ldots, N_k)$,
as before. Then we have  
$$
\lim_{z\to \infty} \left(\Log(z)+\frac{{\bf G}_{\bf N}(z)}{{\bf F}_{\bf N}(z)}\right) 
=-\pi\cot(\pi/M)+i\pi,
$$
where the limit has
to be performed along a path that avoids the cut $[0,+\infty)$.
Here, ${\bf F}_{\bf N}(z)$ and ${\bf G}_{\bf N}(z)$ are given by their
analytic continuations discussed just above, while $\Log(z)$ denotes
the branch of the logarithm described at the beginning of this section.
\end{lem}
\begin{proof} By a well-known method, the integral~\eqref{eq:barnes}
enables us to 
obtain an alternative expression for $F(h,z)$ for $\vert z\vert>1/C$ in the
cut plane 
$\vert \arg(1/C-z)\vert<\pi$: we shift the contour $\mathcal{C}$ to the left, taking into account 
the various poles of the integrand coming from the product 
$\Gamma(\al_1+s)\cdots \Gamma(\al_\Phi+s)$. 

Let us start with the case $k=1$, in which case $\al_j=r_j/N$. Then,
for $\vert z\vert>1/C$ and $\vert \arg(1/C-z)\vert<\pi$, we have
\begin{equation} \label{eq:Fh} 
F(h,z)=
e^{i\pi h}\sum_{j=1}^{\varphi(N)} \frac{\pi\Gamma(1+h)^{\varphi(N)}}
{\sin\Big(\pi(\frac{r_j}N+h)\Big)\prod_{\ell=1}^{\varphi(N)}\Gamma(\frac{r_\ell}{N}+h)} 
\sum_{\ell =0}^{\infty}
\frac{\prod_{\ell=1, \ell\neq j}^{\varphi(N)}\Gamma(\frac{r_\ell-r_j}{N}-\ell )}
{\ell !\,\Gamma(1-\frac{r_j}{N}-\ell )^{\varphi(N)}} (-z)^{-\ell -r_j/N},
\end{equation}
and similarly 
\begin{multline} \label{eq:DFh}
\frac{\partial F}{\partial h}(h,z)
\\
=
\sum_{j=1}^{\varphi(N)}\frac{\partial }{\partial h}\bigg(e^{i\pi h}
\frac{\pi\Gamma(1+h)^{\varphi(N)}}{\sin\Big(\pi(\frac{r_j}N+h)\Big)
\prod_{\ell=1}^{\varphi(N)}\Gamma(\frac{r_\ell}{N}+h)} 
\bigg)  \sum_{\ell =0}^{\infty}
\frac{\prod_{\ell=1, \ell\neq j}^{\varphi(N)}\Gamma(\frac{r_\ell-r_j}{N}-\ell )}
{\ell !\,\Gamma(1-\frac{r_j}{N}-\ell )^{\varphi(N)}} (-z)^{-\ell -r_j/N}.
\end{multline}
In both cases, the leading term 
is the one corresponding to $(-z)^{-1/N}$, and thus
$$
\lim_{z\to\infty} 
\frac{\frac{\partial F}{\partial h}(0,z)}{F(0,z)} = \log(C)-\pi \cot(\pi/N)+i\pi.
$$
(Again, we use~\eqref{eq:logC} to get the value $\log(C)$.) Using~\eqref{eq:FG}, the lemma follows 
in this case.

In the general case, it can be much more complicated to compute precisely the expansions 
because the poles might have 
multiplicity (i.e., some of the $\al$'s might be equal or differ by an integer). The expected 
expansions are linear forms in the functions
$$(-z)^{-\al_\ell}\log^j(-z)g_{\ell,j}(1/z), \quad 
j=0, \ldots, \beta_\ell-1,$$ 
with coefficients that depend on 
$h$. Here $\beta_\ell$ is the multiplicity of $\al_\ell$, and the $g_{\ell,j}(z)$'s are holomorphic 
functions at $z=0$. The main term in the expansions of 
${\bf F}_{\bf N}(z)$ and 
${\bf G}_{\bf N}(z)+\Log(Cz){\bf F}_{\bf N}(z)$ are those corresponding to 
$(-z)^{-1/M}\log^{\beta-1}(-z)$ (with the same maximal $\beta$ in both cases), 
which can be computed without difficulty from~\eqref{eq:barnes}. We get
again
$$
\lim_{z\to\infty} 
\frac{\frac{\partial F}{\partial h}(0,z)}{F(0,z)} = \log(C)-\pi \cot(\pi/M)+i\pi,
$$
and the lemma follows.
\end{proof}

\section{Singular expansions for $\mathbf q_{\mathbf N}(z)$ and
$\mathbf z_{\mathbf N}(q)$}
\label{sec:aux}

The purpose of this section is to discuss the singular expansions of 
$\mathbf q_{\mathbf N}(z)$ at $z=\infty$
(see Lemma~\ref{lem:expinf}) and $z=1/C_{\mathbf N}$
(see Lemma~\ref{lem:exp1/C}),
and of $\mathbf z_{\mathbf N}(q)$ at the ``corresponding" points
$q=-\exp\big(-\pi\cot(\pi/M_{\mathbf N})\big)$ (recall
\eqref{eq:liminfini}) and $q=\mathbf q_{\mathbf N}(1/C_{\mathbf N})$. 
To obtain the latter, one has to combine Lemma~\ref{lem:expinf} with
Lemmas~\ref{lem:qz2}--\ref{lem:qzess}, respectively 
Lemma~\ref{lem:exp1/C} with Lemma~\ref{lem:qz}.

We start with the singular expansion of 
$\mathbf q_{\mathbf N}(z)$ at $z=\infty$.

\begin{lem} \label{lem:expinf}
Let $N_1,N_2, \ldots, N_k$ be positive integers, all of which at least $2$,
and let $M=M_{\mathbf N}$, as before.
Furthermore, let $R$ be the least number different from $1/M$ in the
multiset
$$\mathfrak R=\{r_{i,j}/N_j:i=1,2,\dots,\ph(N_j),\ j=1,2,\dots,k\}.$$

\begin{enumerate}
\item[$(i)$]  
If both $1/M$ and $R$ appear exactly once in the multiset $\mathfrak R$,
then $\mathbf q_{\mathbf N}(z)$ admits a singular
expansion at $\infty$ of the form
\begin{equation} \label{eq:expi} 
\mathbf q_{\mathbf N}(z)=\mathfrak q_0+\mathfrak q_1(-z)^{-R+\frac {1} {M}}+
\mathcal O\Big((-z)^{-R+\frac {1} {M}-\frac {1} {L}}\log^{B_3-1}(-z)\Big),
\end{equation}
where $\mathfrak q_0=-\exp\big(-\pi\cot(\pi/M)\big)$, 
$\mathfrak q_1$ is a non-zero constant, 
$L=\lcm(N_1,N_2,\dots,\break N_k)$, and
$B_3$ is the multiplicity of the third-smallest element in $\mathfrak R$.

\item[$(ii)$]  
If $1/M$ appears exactly once in the multiset $\mathfrak R$,
while $R$ appears with multiplicity $B_2$,
then $\mathbf q_{\mathbf N}(z)$ admits a singular
expansion at $\infty$ of the form
\begin{equation} \label{eq:expii} 
\mathbf q_{\mathbf N}(z)=\mathfrak q_0+\mathfrak q_1(-z)^{-R+\frac {1}
  {M}}\log^{B_2-1}(-z)+
\mathcal O\Big((-z)^{-R+\frac {1}
  {M}}\log^{B_2-2}(-z)\Big),
\end{equation} 
where $\mathfrak q_0$ and
$\mathfrak q_1$ are the same constants as in $(i)$.
\item[$(iii)$]  
If $1/M$ appears with multiplicity at least $2$ in $\mathfrak R$,
then $\mathbf q_{\mathbf N}(z)$ admits a singular
expansion at $\infty$ of the form
\begin{equation} \label{eq:expiii} 
\mathbf q_{\mathbf N}(z)=\mathfrak q_0+\widetilde{\mathfrak q}_1
  \log^{-1}(-z)+
\mathcal O\big(\log^{-2}(-z)\big),
\end{equation} 
where $\mathfrak q_0$ has the same meaning as in $(i)$,
and $\widetilde{\mathfrak q}_1$ is a non-zero constant.
\end{enumerate}
\end{lem}

\begin{Remark}
If the multiset $\mathfrak R$ is in fact a set, i.e., if all elements
of $\mathfrak R$ appear with multiplicity~$1$, then $\mathbf
q_{\mathbf N}(z)$ admits a Puiseux expansion in $(-z)^{-1/L}$, of
which \eqref{eq:expi} shows the first terms. In all other cases, the
singular expansion at $z=\infty$ has terms containing $\log(-z)$.
\end{Remark}

\begin{proof}[Proof of Lemma~\ref{lem:expinf}]
$(i)$
By \eqref{eq:FG} and (in case
$B_3\ge2$: the appropriately generalised) expansions
\eqref{eq:Fh} and \eqref{eq:DFh}, we know that 
\begin{equation} \label{eq:qexp1000} 
\mathbf q_{\mathbf N}(z)=\frac {1} {C}\exp\bigg(\frac {\frac {\partial F}
  {\partial h}(0,Cz)} {F(0,Cz)}\bigg),
\end{equation}
where 
\begin{equation} \label{eq:FB3} 
F(0,Cz)=F_1(-Cz)^{-\frac {1} {M}}+F_2(-Cz)^{- {R}}+\mathcal O\big((-z)^{-
  {R}-\frac {1} {L}}\log^{B_3-1}(-z)\big)
\end{equation}
and
\begin{equation} \label{eq:DFB3} 
{\frac {\partial F} {\partial h}} (0,Cz)=
G_1(-Cz)^{-\frac {1} {M}}+G_2(-Cz)^{- {R}}+\mathcal O
\big((-z)^{- {R}-\frac {1} {L}}\log^{B_3-1}(-z)\big).
\end{equation}
Here, $F_1$, $F_2$, $G_1$, $G_2$ are explicit non-zero constants.
If we use this in \eqref{eq:qexp1000}, then we
obtain 
$$
\mathbf q_{\mathbf N}(z)=\frac {1} {C}\exp\left(
\frac {G_1} {F_1}\left(1+\left(\frac {G_2} {G_1}-\frac {F_2}
    {F_1}\right)(-Cz)^{-R+\frac {1} {M}}+\mathcal O
\big((-z)^{- {R}+\frac {1} {M}-\frac {1} {L}}\log^{B_3-1}(-z)\big)
\right)
\right).
$$
From the explicit expressions and \eqref{eq:logC}, 
it is not difficult to see that
$G_1/F_1=\log(C)-\pi\cot(\pi/M)+i\pi$ and that, furthermore, 
$\frac {G_2} {G_1}\ne\frac {F_2} {F_1}$.
The assertions in $(i)$ now
follow easily upon expansion of the exponential.

\smallskip
$(ii)$ We proceed in the same way as in $(i)$. Here, the expansions 
\eqref{eq:FB3} and \eqref{eq:DFB3} must be replaced by
$$
F(0,Cz)=F_1(-Cz)^{-\frac {1} {M}}+F_2(-Cz)^{-
  {R}}\log^{B_2-1}(-z)+\mathcal O\Big((-z)^{- 
  {R}}\log^{B_2-2}(-z)\Big) 
$$
and
$$
{\frac {\partial F} {\partial h}} (0,Cz)=
G_1(-Cz)^{-\frac {1} {M}}+G_2(-Cz)^{- {R}}\log^{B_2-1}(-z)+\mathcal O
\Big((-z)^{- {R}}\log^{B_2-2}(-z)\Big).
$$
The constants $F_1,F_2,G_1,G_2$ are the same as in $(i)$. The
remaining steps are completely analogous to those in $(i)$ and are
therefore omitted.

\smallskip
$(iii)$
Again, we proceed in the same way as in $(i)$. Here, the expansions 
\eqref{eq:FB3} and \eqref{eq:DFB3} must be replaced by
\begin{multline*}
F(0,Cz)=F_1(-Cz)^{-\frac {1} {M}}\log^{B_1-1}(-z)+\widetilde F_2(-Cz)^{-
\frac 1  {M}}\log^{B_1-2}(-z)\\+\mathcal O\Big((-z)^{- 
  \frac 1{M}}\log^{B_2-3}(-z)\Big) 
\end{multline*}
and
\begin{multline*}
{\frac {\partial F} {\partial h}} (0,Cz)=
G_1(-Cz)^{-\frac {1} {M}}\log^{B_1-1}(-z)+
\widetilde G_2(-Cz)^{- \frac1{M}}\log^{B_1-2}(-z)\\+\mathcal O
\Big((-z)^{-\frac 1 {M}}\log^{B_1-3}(-z)\Big).
\end{multline*}
The constants $F_1,F_2$ are the same as in $(i)$. The
remaining steps are completely analogous to those in $(i)$ and are
therefore omitted.
\end{proof}

The next lemma addresses the singular expansion of 
$\mathbf q_{\mathbf N}(z)$ at $z=1/C_{\mathbf N}$.

\begin{lem} \label{lem:exp1/C}
Let $N_1,N_2, \ldots, N_k$ be positive integers, all of which at least $2$,
and let $C=C_{\mathbf N}$, $\Phi=\Phi_{\mathbf N}$, as before.

\begin{enumerate}
\item[$(i)$]  
If $\Phi=2$,
then $\mathbf q_{\mathbf N}(z)$ admits a singular
expansion at $1/C$ of the form
\begin{equation} \label{eq:1/C2} 
\mathbf q_{\mathbf N}(z)=1+{\sf q}_1\log^{-1}(1-Cz)+
\mathcal O\big(\log^{-2}(1-Cz)\big),
\end{equation}
where ${\sf q}_1>0$.

\item[$(ii)$]  
If $\Phi\ge3$ is odd,
then $\mathbf q_{\mathbf N}(z)$ admits a singular
expansion at $1/C$ of the form
\begin{multline} \label{eq:1/Cimpair} 
\mathbf q_{\mathbf
  N}(z)=\mathbf q_{\mathbf
  N}(1/C)+{\sf q}_1(1-Cz)+{\sf q}_2(1-Cz)^2+\cdots\\
+{\sf q}_d(1-Cz)^d+
{\sf q}_{d+\frac {1} {2}}(1-Cz)^{d+\frac {1} {2}}
+\mathcal O\big((1-Cz)^{d+1}\big),
\end{multline}
where $d=\frac {\Phi-3} {2}$, 
${\sf q}_1<0$, and $(-1)^{d+1}{\sf q}_{d+\frac {1} {2}}>0$.

\item[$(iii)$]  
If $\Phi\ge4$ is even,
then $\mathbf q_{\mathbf N}(z)$ admits a singular
expansion at $1/C$ of the form
\begin{multline} \label{eq:1/Cpair} 
\mathbf q_{\mathbf
  N}(z)=\mathbf q_{\mathbf
  N}(1/C)+{\sf q}_1(1-Cz)+{\sf q}_2(1-Cz)^2+\cdots\\
+{\sf q}_{d-1}(1-Cz)^{d-1}+
{\sf q}_{d+}(1-Cz)^{d}\log(1-Cz)
+\mathcal O\big((1-Cz)^{d}\big),
\end{multline}
where $d=\frac {\Phi-2} {2}$, 
${\sf q}_1<0$, and $(-1)^{d+1}{\sf q}_{d+}>0$.
\end{enumerate}
\end{lem}

\begin{proof}
We proceed by using the theory of hypergeometric differential
equations to determine the form of the singular expansion of
the quotient ${\bf G}_{\bf N}(z)/{\bf F}_{\bf N}(z)$,
see \eqref{eq:G}. 
This is then translated in the final step into the claimed singular
expansions for $\mathbf q_{\mathbf N}(z)$.

First of all, from Section~\ref{sec:analytic}
we know that ${\bf F}_{\bf N}(z)$ and ${\bf G}_{\bf N}(z) + 
\log(az){\bf F}_{\bf N}(z)$ (for any $a\neq 0$) 
 can be analytically continued to 
$\mathbb{C}\setminus[1/C,+\infty)$ and 
$\mathbb{C}\setminus\big((-\infty,0]\cup [1/C,+\infty)\big)$, 
respectively. 
We want to determine their behaviour around the point $z=1/C$. 
The exponents at the regular singular point 
$z=1/C$ of the differential equation ${\bf L}y=0$
(with $\mathbf L$ being defined in \eqref{eq:equadiff}) 
are $0,1,\ldots, \Phi-2$, and 
\begin{equation} \label{eq:expon} 
(\Phi-1)-\sum_{j=1}^k \sum_{i=1}^{\varphi(N_j)} \frac{r_{i,j}}{N_j}. 
\end{equation}
By the elementary identity
\begin{equation} \label{eq:Phi/2}
\sum_{j=1}^k \sum_{i=1}^{\varphi(N_j)} \frac{r_{i,j}}{N_j}=
\frac{\Phi}2, 
\end{equation}
the value \eqref{eq:expon} simplifies to $\frac {\Phi} {2}-1$.
By the theory of hypergeometric differential equations (cf.\
\cite[Ch.~4, Sec.~8]{CoLeAA}), 
a basis over $\mathbb{C}$ of solutions of 
${\bf L}$ consists of $\Phi-1$ functions $f_1(z), \ldots, f_{\Phi-1}(z)$ holomorphic at 
$z=1/C$, together with 
another solution $f_{\Phi}(z)$ which can be described as follows:
\begin{itemize}
\item[a)] if $\Phi$ is odd, then $f_{\Phi}(z)=(1-Cz)^{\Phi/2-1}u(z)$, 
where $u(z)$ is holomorphic at $z=1/C$;
\item[b)] if $\Phi$ is even, then $f_{\Phi}(z)=v(z)+(1-Cz)^{\Phi/2-1}\log(1-Cz)
u(z)$, where both $u(z)$ and $v(z)$ are holomorphic at $z=1/C$.
\end{itemize}

It follows that, in a neighbourhood of $1/C$ avoiding the cut $[1/C,+\infty)$, 
we have~(\footnote{Since the function
${\bf G}_{\bf N}(z)+\log(az){\bf F}_{\bf N}(z)$ (with $a\neq 0$) is a solution 
of ${\bf L}y=0$, it can be written in a form similar to 
the right-hand side of~\eqref{eq:F01}. Application of~\eqref{eq:F01} to 
$\log(az){\bf F}_{\bf N}(z)$ then gives~\eqref{eq:G01} because 
$\log(az)$ is holomorphic at $z=1/C$.})
\begin{align}
{\bf F}_{\bf N}(z) &= f(z) + (1-Cz)^{\Phi/2-1} L(z) g(z) \label{eq:F01}
\\
{\bf G}_{\bf N}(z)&= \breve{f}(z) + (1-Cz)^{\Phi/2-1} L(z) \breve{g}(z), \label{eq:G01}
\end{align}
where $f,\breve{f}, g$ and $\breve{g}$ are holomorphic around $z=1/C$ 
and  $L(z)=1$ if 
$\Phi$ is odd, 
respectively $L(z)=-\log(1-Cz)$ if $\Phi$ is even.

Concerning the coefficients of ${\bf F}_{\bf N}(z)$, by Stirling's formula, we have 
\begin{equation} \label{eq:Basy} 
{\bf B}_{\bf N}(m)= \bigg(
\prod _{j=1} ^{k}
\prod _{i=1} ^{\varphi(N_j)}\frac {1} {\Ga(r_{i,j}/N_j)}\bigg)
\cdot
\frac{C^m}{m^{\Phi/2}}\left(1+o(1)\right),\quad \quad m\to\infty,
\end{equation}
which implies that $g(1/C)> 0$ by the classical link between singularities of an analytic 
function $h$ 
and the asymptotic behaviour of the Taylor coefficients of $h$ when 
these are positive (see \cite[Ch.~VI]{FlSeAA}).

Concerning the coefficients of ${\bf G}_{\bf N}(z)$, we have 
\begin{equation}\label{eq:asympH}
{\bf B}_{\bf N}(m) {\bf H}_{\bf N}(m) = {\bf B}_{\bf N}(m) \bigg(\log(C)
-\frac{\Phi}{2m} + \mathcal{O}\Big(\frac1{m^2}\Big)\bigg).
\end{equation}
(In fact, 
$\log(C)$ appears under the form $\sum_{j=1}^k \sum_{i=1}^{\varphi(N_j)}
\big(\psi(1)-\psi(r_{i,j}/N)\big)$, see \eqref{eq:logC}.)
Hence, using~\eqref{eq:asympH}, we can make~\eqref{eq:G01} more 
precise~(\footnote{This is more precise when one transforms $\log(C){\bf F}_{\bf N}(z)$
using~\eqref{eq:F01}.}): 
\begin{equation}\label{eq:Gprecise}
{\bf G}_{\bf N}(z) = \log(C){\bf F}_{\bf N}(z)
+\widetilde{f}(z) + (1-Cz)^{\Phi/2} L(z) \widetilde{g}(z),
\end{equation}
where $\widetilde{f}$ and $\widetilde{g}$ are holomorphic at $z=1/C$,
and $\widetilde g(1/C)\ne0$. 

In order to proceed, we need the following auxiliary result.

{\leftskip1cm\rightskip1cm\noindent\it
For any vector ${\bf N}$ of positive integers, the limit 
\begin{equation} \label{eq:lem3} 
S:=\lim_{z \to 1/C} \big({\bf G}_{\bf N}(z)-\log(C){\bf F}_{\bf N}(z)\big)
\end{equation}
exists, is finite and is $<0$. Furthermore, it is equal to $\widetilde
f(1/C)$.\par}

Above and in the sequel, the limit $z \to 1/C$ is understood along real numbers $z<1/C$.

In order to see \eqref{eq:lem3}, we observe that,
for $\vert z\vert <1/C$, we have 
$$
{\bf G}_{\bf N}(z)-\log(C){\bf F}_{\bf N}(z) = \sum_{m=0}^{\infty} {\bf B}_{\bf N}(m)  
\big({\bf H}_{\bf N}(m)-\log(C)\big) z^m.
$$
The series on the right-hand side converges for $z=1/C$ because 
$$
\big\vert {\bf B}_{\bf N}(m)  
\big({\bf H}_{\bf N}(m)-\log(C)\big)\big\vert= \mathcal O\left( \frac{C^m}{m^{\Phi/2+1}}
\right)$$
and $\Phi/2+1>1$.  By Abel's theorem, the limit $S$
in \eqref{eq:lem3} exists and 
$$
S=  \sum_{m=0}^{\infty} {\bf B}_{\bf N}(m)  
\big({\bf H}_{\bf N}(m)-\log(C)\big) C^{-m},
$$
the right-hand side being finite.

Secondly, since $H(x,n)=\psi(n+x)-\psi(x)$, it is easy to see that 
$$
{\bf H}_{\bf N}(m)-\log(C) = \sum_{j=1}^k\sum_{i=1}^{\varphi(N_j)} 
\Big(\psi\big(m+ \tfrac{r_{i,j}}{N_j}\big)-\psi(m+1)\Big),
$$
where we used \eqref{eq:logC} again.
Since the function $\psi$ is strictly increasing on $(0, +\infty)$, and 
since $0<\frac{r_{i,j}}N<1$, we deduce that 
\begin{equation} \label{eq:HlogC} 
{\bf H}_{\bf N}(m)-\log(C)<0\quad \text{for all }m\ge 0.
\end{equation}
Hence $S<0$.

Finally, since $\Phi/2>0$, 
Eq.~\eqref{eq:Gprecise} implies that 
$$
\lim_{z\to 1/C} \big({\bf G}_{\bf N}(z) - \log(C){\bf F}_{\bf N}(z)\big) 
= \lim_{z\to 1/C} \big(\widetilde{f}(z) +  (1-Cz)^{\Phi/2} L(z) \widetilde{g}(z) \big) =\widetilde{f}(1/C),
$$
thus completing the proof of \eqref{eq:lem3}.

\medskip
We may now continue with the proof of the lemma.
By the remarks preceding \eqref{eq:lem3}, we have
\begin{equation}\label{eq:G}
\frac{{\bf G}_{\bf N}(z)}{{\bf F}_{\bf N}(z)} = \log(C) + 
\frac{\widetilde{f}(z) + (1-Cz)^{\Phi/2} L(z) \widetilde{g}(z)}{f(z) + (1-Cz)^{\Phi/2-1} L(z) g(z)}.
\end{equation}

This is now translated to ${\bf q}_{\bf N}(z)=
z\exp\big({\bf G}_{\bf N}(z)/{\bf F}_{\bf N}(z)\big)$.
Let us for the moment restrict ourselves to the case $\Phi\ge3$.
It was argued in the paragraph
between \eqref{eq:G01} and \eqref{eq:asympH} that
$g(1/C)>0$. 
Furthermore, by
the definition of $f(z)$ in \eqref{eq:F01}, we have
$f(1/C)={\bf F}_{\bf N}(1/C)\ne0$. Finally, by \eqref{eq:lem3},
we have $\widetilde f(1/C)<0$. 
If we use these observations, together with our
assumption that $\Phi\ge3$, from \eqref{eq:G} we obtain the singular
expansion 
\begin{align} \notag
{\bf q}_{\bf N}(z)&=z\exp\left(\frac {{\bf G}_{\bf N}(z)} {{\bf F}_{\bf
      N}(z)}\right) \\
\label{eq:singexp}
&=
Cz\exp\big(\widetilde f(z)/f(z)\big)
\Big(1+\alpha_1(1-Cz)^{\Phi/2-1}L(z)+\mathcal O\big((1-Cz)^{\Phi/2}L(z)\big)\Big)
\end{align}
for $z\to 1/C$. Here, 
$$
\alpha_1=-\exp\bigg(\frac {\widetilde f(1/C)} {f(1/C)}\bigg)
\frac {g(1/C)\widetilde f(1/C)}
{f^2(1/C)}=
{\bf q}_{\bf N}(1/C)\frac 
{g(1/C)\big({\bf F}_{\bf N}(1/C)\log(C)-{\bf G}_{\bf N}(1/C)\big)}
{{\bf F}^2_{\bf N}(1/C)},
$$ 
which is positive because of \eqref{eq:HlogC} and $g(1/C)>0$.
The singular expansions \eqref{eq:1/Cimpair} and
\eqref{eq:1/Cpair} now follow routinely: the claim on the sign of
${\sf q}_{d+\frac {1} {2}}$, respectively of
${\sf q}_{d+}$, is a direct consequence of $\al_1$ being
positive, while the claim on the sign of ${\sf q}_1$ follows from
the fact that $\mathbf q_{\mathbf N}(z)$ is monotone increasing on the
interval $[0,1/C]$ (which is implied by Theorem~\ref{theo:1}). 

Finally, if $\Phi=2$, then, leaving the details to the reader, the 
singular expansion in this case is
\begin{equation} \label{eq:singexp4}
{\bf q}_{\bf N}(z)
=
1+\frac {\alpha_2} {\log(1-Cz)}+\mathcal O\left(\frac {1} {\log^2(1-Cz)}\right)
\end{equation}
for $z\to 1/C$. 
Here,
$
\alpha_2=- {\widetilde f(1/C)} \big/{g(1/C)},
$ 
which is positive by our earlier observations.
\end{proof}

The remaining lemmas in this section are general results that
describe the singular expansion of a function $z(q)$ at $q=q_0$,
given the singular expansion of its compositional inverse $q(z)$ at
the corresponding point. They are tailor-made for obtaining 
singular expansions of $\mathbf z_{\mathbf N}(q)$ at 
$q=-\exp\big(-\pi\cot(\pi/M)\big)$, respectively at
$q=\mathbf q_{\mathbf N}(1/C)$, by combining the appropriate lemma
with Lemma~\ref{lem:expinf}, respectively with
Lemma~\ref{lem:exp1/C}.

We start with the results which, upon combination with
Lemma~\ref{lem:exp1/C}, imply singular expansions of 
$\mathbf z_{\mathbf N}(q)$ at $q=\mathbf q_{\mathbf N}(1/C)$.
In the statements, we make use of right and left slit
neighbourhoods, notions that have been defined in the paragraph above 
Proposition~\ref{prop:1} in the Introduction.

\begin{lem} \label{lem:qz2}
Suppose we are given a complex function $q(z)$ which is analytic in a 
right slit neighbourhood of $z=z_0$ and
has a singular expansion that begins
\begin{multline} \label{eq:qsing}
q(z)=q_0+q_1(z-z_0)+q_2(z-z_0)^2+\cdots\\+q_d(z-z_0)^d+
q_{d+\frac {1} {2}}(z_0-z)^{d+\frac {1} {2}}
+\mathcal O\big((z-z_0)^{d+1}\big),
\end{multline}
where $d\ge0$, $q_{d+\frac {1} {2}}\ne0$ and, if $d>0$, then also $q_1\ne0$.
Then there exists a local inverse function $z(q)$ which, in a
right slit neighbourhood of $q=q_0$, is analytic and
has a singular expansion that begins
\begin{equation*}
z(q)=z_0+z_2(q-q_0)^2
+\mathcal O\big((q-q_0)^{4}\big)
\end{equation*}
if $d=0$, where $z_{2}=-1/q_{1/2}^2$,
and begins
\begin{multline*}
z(q)=z_0+z_1(q-q_0)+z_2(q-q_0)^2+\cdots\\+z_d(q-q_0)^d+
z_{d+\frac {1} {2}}(q_0-q)^{d+\frac {1} {2}}
+\mathcal O\big((q-q_0)^{d+1}\big)
\end{multline*}
if $d>0$,
where $z_1=1/q_1$ and $z_{d+\frac {1} {2}}=-q_1^{-d-\frac {3} {2}}
q_{d+\frac {1} {2}}$.
\end{lem}

\begin{proof}
We concentrate on the case where $d>0$.
By standard bootstrap arguments (see 
\cite[Sections~2.5--2.7]{BruiAA}), one sees that, in a right slit
neighbourhood of $q_0$,
\begin{equation} \label{eq:ztrunc} 
z(q)=z_0+z_1(q-q_0)+z_2(q-q_0)^2+\cdots+z_d(q-q_0)^d+\tilde z(q),
\end{equation}
where the coefficients $z_0,z_1,\dots,z_d$ agree with the corresponding Taylor
coefficients of the compositional inverse of the truncated series
$$
q_0+q_1(z-z_0)+q_2(z-z_0)^2+\cdots+q_d(z-z_0)^d,
$$
and where $\tilde z(q)$ is a function which is analytic in the same
slit neighbourhood, and which satisfies $\tilde z(q)=
o\big((q-q_0)^d\big)$. In particular, $z_1=1/q_1$.
Continuing the bootstrap,
the expansion \eqref{eq:ztrunc} (with $q$ replaced by $x$ in order to
minimise the number of possible confusions) is now
substituted in \eqref{eq:qsing}. In that manner, we obtain
$$
0=q_1\tilde z(x)+q_{d+\frac {1} {2}}z_1^{d+\frac {1}
  {2}}(q_0-x)^{d+\frac {1} {2}} +\mathcal O\big((x-q_0)^{d+1}\big).
$$
The claimed result is now obvious.
\end{proof}

\begin{lem} \label{lem:qz3}
Suppose we are given a complex function $q(z)$ which is analytic in a
right slit neighbourhood of $z=z_0$, where it
has a singular expansion that begins
\begin{multline} \label{eq:qsinglog}
q(z)=q_0+q_1(z-z_0)+q_2(z-z_0)^2+\cdots\\+q_{d-1}(z-z_0)^{d-1}+
q_{d+}(z-z_0)^{d}\log(z_0-z)
+\mathcal O\big((z-z_0)^{d}\big),
\end{multline}
where $d\ge1$, $q_{d+}\ne0$, and, if $d>1$, then also $q_1\ne0$.
Then there exists a local inverse function $z(q)$ which, in a
right slit neighbourhood of $q=q_0$, is analytic and
has a singular expansion that begins
\begin{equation} \label{eq:zsinglog2}
z(q)=z_0+z_{1+}(q-q_0)\log^{-1}(q_0-q)
+o\big((q-q_0)\log^{-1}(q_0-q)\big)
\end{equation}
if $d=1$, where $z_{1+}=1/q_{1+}$,
and begins
\begin{multline} \label{eq:zsinglog1}
z(q)=z_0+z_1(q-q_0)+z_2(q-q_0)^2+\cdots\\+z_{d-1}(q-q_0)^{d-1}+
z_{d+}(q-q_0)^{d}\log(q_0-q)
+\mathcal O\big((q-q_0)^{d}\big)
\end{multline}
if $d>1$, where $z_1=1/q_1$ and $z_{d+}=-q_1^{-d-1}q_{d+}$.
\end{lem}

\begin{proof}
For \eqref{eq:zsinglog1}, one proceeds exactly in the same fashion as
in the proof of Lemma~\ref{lem:qz2}.

In order to establish the expansion \eqref{eq:zsinglog2}, one 
replaces $z$ by $z(x)$ in \eqref{eq:qsinglog}, and subsequently one
applies the logarithm on both sides. This leads to
\begin{equation*} 
\log(q_0-x)=\log (q_{1+})+\log(z_0-z(x))+\log\log(z_0-z(x))+
\mathcal O\big(\log^{-1}(z_0-z(x))\big).
\end{equation*}
In order to simplify, we replace $q_0-x$ by $X$ and $\log(z_0-z(x))$ by $Z(X)$:
\begin{equation} \label{eq:boot1} 
\log(X)=\log (q_{1+})+Z(X)+\log(Z(X))+
\mathcal O\big(Z(X)^{-1}\big).
\end{equation}
Applying bootstrapping again in a neighbourhood of $X=0$, 
we must have $Z(X)=\log(X)+\tilde Z(X)$, where $\tilde
Z(X)=o(\log(X))$. If we substitute this in \eqref{eq:boot1}, we
obtain, after little simplification,
\begin{equation} \label{eq:boot2} 
0=\log (q_{1+})+\tilde Z(X)+\log\log(X)+o(1).
\end{equation}
Now we see that $\tilde Z(X)=-\log\log(X)+\tilde{\tilde Z}(X)$,
where $\tilde{\tilde Z}(X)=o(\log\log(X))$. By substituting this in
\eqref{eq:boot2}, we arrive at
\begin{equation} \label{eq:boot3} 
0=\log (q_{1+})+\tilde {\tilde Z}(X)+o(1),
\end{equation}
from which we deduce $\tilde {\tilde Z}(X)=-\log (q_{1+})+o(1)$.
If we now put everything together, then we obtain
$$
Z(X)=\log(X)-\log\log(X)-\log (q_{1+})+o(1),
$$
or, in the original notation,
$$
z_0-z(x)=\frac {q_0-x} {q_{1+}\log(q_0-x)}\big(1+o(1)\big).
$$
After replacement of $x$ by $q$, one sees that this is equivalent to
\eqref{eq:zsinglog2}. 
\end{proof}

\begin{lem} \label{lem:qzess}
Suppose we are given a complex function $q(z)$ which is analytic in a 
right slit neighbourhood of $z=z_0$, where it
has a singular expansion that begins
$$
q(z)=q_0+{q_1} {\log^{-1}(z_0-z)}+
\mathcal O\big({\log^{-2}(z_0-z)}\big),
$$
where $q_1\ne0$.
Then there exists a local inverse function $z(q)$ which, in a
right slit neighbourhood of $q=q_0$, is analytic and
has a singular expansion of the form
$$
z(q)=\exp\bigg(\frac {q_1} {q-q_0}+\mathcal O(1)\bigg).
$$
In particular, $z(q)$ has an essential singularity at $q=q_0$.
\end{lem}

\begin{proof}
This is again easily derived by bootstrapping.
\end{proof}

Finally we provide a general result which, upon combination with
Lemma~\ref{lem:expinf}, yields the singular expansion 
of $\mathbf z_{\mathbf N}(q)$ at 
$q=-\exp\big(-\pi\cot(\pi/M)\big)$.

\begin{lem} \label{lem:qz}
Let $q(z)$ be a complex function that is analytic near $\infty$ except
for a cut on the positive real axis.

\begin{enumerate}
\item[$(i)$]  
Suppose that $q(z)$ has a singular expansion at
$z=\infty$ that begins
$$
q(z)=q_0+q_1(-z)^{-\frac {r} {n}}+
\mathcal O\Big((-z)^{-\frac {r+1} {n}}\log^b(-z)\Big),
$$
where $q_1\ne0$ and $r,n,b$ are non-negative integers such that $r\ge1$.
Then there exists a local inverse function $z(q)$ which, in a
neighbourhood of $q=q_0$, has a singular expansion that begins
$$
z(q)=-q_1^{\frac {n} {r}}(q-q_0)^{-\frac {n} {r}}+
\mathcal O\Big((q-q_0)^{-\frac {n-1} {r}}\log^{b}\big(q_0-
{q} \big)\Big).
$$

\item[$(ii)$]  
Suppose that $q(z)$ has a singular expansion at
$z=\infty$ that begins
$$
q(z)=q_0+q_1(-z)^{-\frac {r} {n}}\log^b(-z)+
\mathcal O\Big((-z)^{-\frac {r} {n}}\log^{b-1}(-z)\Big),
$$
where $q_1\ne0$ and $r,n,b$ are non-negative integers with $r,b\ge1$.
Then there exists a local inverse function $z(q)$ which, in a
neighbourhood of $q=q_0$, has a singular expansion that begins
$$
z(q)=-\left(-\tfrac {n} {r}\right)^{\frac {br} n}q_1^{\frac {n} {r}}
(q-q_0)^{-\frac {n} {r}}\log^{\frac {bn} r}\big(q_0-
{q}\big)+
\mathcal O\Big((q-q_0)^{-\frac {n} {r}}\log^{\frac {bn} {r}-1}\big(q_0-
{q} \big)\Big).
$$

\item[$(iii)$]  
Suppose that $q(z)$ has a singular expansion at
$z=\infty$ that begins
$$
q(z)=q_0+q_1\log^{-1}(-z)+
\mathcal O\Big(\log^{-2}(-z)\Big),
$$
where $q_1\ne0$.
Then there exists a local inverse function $z(q)$ which, in a
neighbourhood of $q=q_0$, has a singular expansion that begins
$$
z(q)=\exp\bigg(\frac {q_1} {q-q_0}+
\mathcal O(1)\bigg).
$$
\end{enumerate}
\end{lem}

\begin{Remark}
If, in Case $(i)$ with $r=1$, the series $q(z)$ admits in fact a Puiseux
expansion in $(-z)^{-\frac {1} {n}}$, then it is not difficult to see
that $z(q)$ has a pole of order $n$ at $q=q_0$. If the singular
expansion of $q(z)$ at $z=-\infty$ should have terms containing
$\log(-z)$ (which is often the case in the situation of
Lemma~\ref{lem:expinf}; see the remark accompanying that lemma), then
the point $q=q_0$ will be a branch point of $z(q)$.
\end{Remark}

\begin{proof}
Once more, this is easily derived by bootstrapping.
\end{proof}

\section{Proof of Theorem~\ref{theo:2}}
\label{sec:theo2}

For the proof of the theorem, we shall require the following
auxiliary result.

\begin{lem} \label{lem:F0}
The function ${\bf F}_{\bf N}(z)$ does not vanish on its disk of 
convergence $\vert z\vert< 1/C_{\bf N}$.
\end{lem}

\begin{proof}

The following argument is borrowed from~\cite[page 94, last corollary]{lamperti}.
By Corollary~\ref{cor:1}, the Taylor coefficients $\widehat{\bf B}_{\bf N}(m)$ 
of $\widehat{\bf F}_{\bf N}(z)$ are positive and satisfy Equation~\eqref{eq:fn}, i.e.,
$$
\delta_{n,0}+\sum_{k=1}^{n} 
\widehat{\bf B}_{\bf N}(k) {\bf B}_{\bf N}(n-k)={\bf B}_{\bf N}(n).
$$
Since the coefficients ${\bf  B}_{\bf N}(m)$ are also positive and 
${\bf  B}_{\bf N}(0)=1$, 
it follows that
$
0 \le \widehat{\bf B}_{\bf N}(n) \le 
{\bf B}_{\bf N}(n)
$
for all $n\ge 0$. Hence, the radius of convergence of the Taylor series of 
$1/{\bf F}_{\bf N}(z)$ at $z=0$ is at least as large as 
the radius of convergence of ${\bf F}_{\bf N}(z)$ at 
$z=0$. It is necessarily equal to the latter because $z=1/C$ is a branch 
point of ${\bf F}_{\bf N}(z)$, and thus also of $1/{\bf F}_{\bf N}(z)$. 

In particular, ${\bf F}_{\bf N}(z)$ cannot vanish at some point $z_0$ such that 
$\vert z_0\vert < 1/C$, because otherwise the radius of convergence of 
$1/{\bf F}_{\bf N}(z)$ would be at most $\vert z_0\vert$, a contradiction.
\end{proof}

\begin{Remark}
In the proof of Theorem~\ref{theo:3}.$(i)$ (given in
Section~\ref{sec:polya}), we show that a classical result of P\'olya
\cite{polya} on hypergeometric series implies that ${\bf F}_{\bf N}(z)$
vanishes nowhere in the slit plane $\mathbb{C}\setminus \big[1/C,
+\infty)$. Nevertheless, we believe that that the above argument,
proving a weaker result, is still worth being recorded since it is
only based on the positivity of the coefficients $\mathbf B_{\mathbf
N}(m)$ and $\widehat{\mathbf B}_{\mathbf N}(m)$ and not on the
hypergeometric nature of ${\bf F}_{\bf N}(z)$.
\end{Remark}

We now turn to the proof of items $(i)$--$(vi)$ of
Theorem~\ref{theo:2}.

\medskip

$(i)$, $(ii)$ In view of the explicit expression \eqref{eq:Phi=1}, the claim
is trivial for $\Phi=1$. We therefore assume $\Phi\ge2$ from now on.

Clearly, the discussion in the proof of Lemma~\ref{lem:F0} 
also implies that the radius of convergence of 
the Taylor series at $z=0$ 
of $\exp({\bf G}_{\bf N}(z)/{\bf F}_{\bf N}(z))$ 
is at least $1/C$. By Lemma~\ref{lem:exp1/C}, which says in particular
that ${\bf q}_{\bf N}(z)$ has a singularity at $z=1/C$, it cannot
be larger. 

It remains to prove that the Taylor series $\sum_{m\ge 1} {\sf q}_mz^m$ 
of ${\bf q}_{\bf N}(z)$ converges on 
the circle $\vert z\vert =1/C$. By \eqref{eq:lem3}, we have  
\begin{equation}\label{eq:limG/F}
\lim_{z\to 1/C} \frac{{\bf G}_{\bf N}(z)}{{\bf F}_{\bf N}(z)} 
= \log(C) + \lim_{z\to 1/C} \frac{{\bf G}_{\bf N}(z)-\log(C){\bf F}_{\bf N}(z)}
{{\bf F}_{\bf N}(z)} = \log(C) + \frac{S}{{\bf F}_{\bf N}(1/C)},
\end{equation}
where the second term must be understood as $0$ if $\lim_{z\to 1/C}{\bf F}_{\bf N}(z) =+\infty$. 
Hence ${\bf q}_{\bf N}(1/C)$ exists and is finite. 
By Abel's theorem, ${\bf q}_{\bf N}(1/C)=\sum_{m\ge 1} {\sf q}_m /C^{m}$, 
and since the ${\sf q}_m$ are 
non-negative, $\sum_{m\ge 1} {\sf q}_mz^m$ converges for any $z$ such that $\vert z\vert =1/C$.

\medskip

$(iii)$ We have
$$
\max_{\vert z\vert =1/C} \vert {\bf q}_{\bf N}(z) \vert = {\bf q}_{\bf N}(1/C).
$$
By~\eqref{eq:limG/F}, we have 
$$
{\bf q}_{\bf N}(1/C) = \exp\Big(\frac{S}{{\bf F}_{\bf N}(1/C)}\Big)\le 1
$$
because $S<0$ and ${\bf F}_{\bf N}(1/C)>0$. There is equality to $1$ only if 
${\bf F}_{\bf N}(1/C)=+\infty$, which, by \eqref{eq:Basy}, happens
exactly when $\Phi=1$ or $\Phi=2$.

\medskip
$(iv)$ 
The claimed assertions follow from \eqref{eq:Phi=1} upon little calculation.

\medskip
$(v)$
By the standard theorems of singularity analysis (see
\cite[Ch.~VI]{FlSeAA}), 
the assertion follows immediately from Lemma~\ref{lem:exp1/C}.$(i)$.

\medskip
$(vi)$
Again, by the standard theorems of singularity analysis, 
the assertion follows immediately from  
Lemma~\ref{lem:exp1/C}.$(ii)$, $(iii)$.
\hfill\qed

\section{Proof of Theorem~\ref{theo:3}}
\label{sec:polya}

$(i)$
From Section~\ref{sec:analytic},
we know that ${\bf F}_{\bf N}(z)$ and ${\bf G}_{\bf N}(z)$  can both be analytically 
continued to $\mathbb{C}\setminus [1/C, +\infty)$. Hence 
$$
{\bf q}_{\bf N}(z)=\exp\Big(\frac{{\bf G}_{\bf N}(z)+\log(z){\bf F}_{\bf N}(z)}{{\bf F}_{\bf N}(z)}\Big) 
=z\exp\Big(\frac{{\bf G}_{\bf N}(z)}{{\bf F}_{\bf N}(z)}\Big)
$$
can be continued at least to 
$\mathbb{C}\setminus \big([1/C, +\infty)\cup Z\big)$, where $Z$ 
is the set of zeroes of ${\bf F}_{\bf N}(z)$. 

The reader should recall that Lemma~\ref{lem:F0} says that the
intersection of $Z$ and the open disk of convergence of ${\bf F}_{\bf
N}(z)$ is empty.
We now show that, in fact, the entire set $Z$ is empty. For this, we apply a result of
P\'olya~\cite[p.~192]{polya} on hypergeometric functions.
Recalling the hypergeometric notation
\begin{eqnarray*}
{}_{q+1} F _{q} \!
\left [ 
\begin{matrix} {a_0, a_1, \ldots,a_q}\\ 
{ b_1, \ldots, b_q}\end{matrix} ; {\displaystyle z}
\right ]
=\sum_{k=0}^{\ii}
\frac{(a_0)_k\,(a_1)_k\cdots(a_{q})_k}
{k!\,(b_1)_k\cdots(b_q)_k} z^k,
\end{eqnarray*}
P\'olya's result implies in particular that the above hypergeometric function 
does not vanish for any $z\in \mathbb{C}\setminus [1, +\infty)$ when $0<a_0<1$, 
$0<a_1<b_1, \ldots, 0<a_{q}<b_{q}$. 
Now, indeed, ${\bf F}_{\bf N}(z)$ can be written in
hypergeometric notation:
$$
{\bf F}_{\bf N}(z)=
{}_{\Phi}F_{\Phi-1}\!\left [ \begin{matrix}
r_{1,1}/N_1,\dots,r_{k,\varphi(N_k)}/N_k\\
1,\dots,1\end{matrix} ;
{\displaystyle Cz}\right ].
$$
In particular, we see that
P\'olya's conditions are satisfied by this
hypergeometric function, which proves that $Z$ is empty.

\medskip
$(ii)$
If $\Phi\ge2$, this is a consequence of the singular expansion of
$\mathbf q_{\mathbf N}(z)$ at $z=1/C$ given in Lemma~\ref{lem:exp1/C}.

If $\Phi=1$, then we know that ${\bf
  q}_{(2)}(z)=(1-\sqrt{1-4z})^2/(4z)$, which has evidently a branch
point with non-trivial monodromy at $z=1/C_{(2)}=1/4$.

\medskip
$(iii)$
The assertion concerning the limit of ${\bf q}_{\bf N}(z)$  at infinity 
is an immediate consequence of Lemma~\ref{eq:asympG/F}.
\hfill\qed

\section{Proof of Theorem~\ref{theo:4}}
\label{sec:rad}

$(i)$
We know from Theorem~\ref{theo:3} that 
$$
\lim_{z\to \infty}{\bf q}_{\bf N}(z) =-\exp(-\pi\cot(\pi/M))=:\rho,
$$
where the limit has
to be performed along a path that avoids the cut $[1/C,+\infty)$.
Let us suppose that the radius of convergence of ${\bf z}_{\bf
N}(q)$, $R$ say, is strictly larger than $\vert\rho\vert$. One can find 
$\ep>0$ and $A(\ep)>0$ with the property that, if $\vert
x\vert>A(\ep)$ and $x\notin [1/C,+\infty)$, then 
$$\vert \mathbf q_{\mathbf
  N}(x)-\rho\vert<\ep
\quad \text{and}\quad \vert\rho\vert+\ep<R.$$
For the above $x$, the quantity ${\bf z}_{\bf N}({\bf q}_{\bf N}(x))$
is well-defined and equals $x$. 
Consequently, 
$$
\infty=
\underset{x\notin [1/C,+\infty)}{\underset{\vert x\vert>A(\ep)}{\lim_{x\to
      \infty}}} 
x =
\underset{x\notin [1/C,+\infty)}{\underset{\vert x\vert>A(\ep)}{\lim_{x\to
      \infty}}} 
{\bf z}_{\bf N}({\bf q}_{\bf N}(x)) = \lim_{q\to \rho} {\bf z}_{\bf N}(q),
$$
where the last limit is along a suitable path. Hence, the point $q=\rho$
is a singularity of ${\bf z}_{\bf N}(q)$, which contradicts our
assumption that $R>\vert\rho\vert$. Therefore the radius of convergence $R$ is
in fact $\le \vert\rho\vert$.\hfill\qed

\medskip
$(ii)$
Arguing by contradiction, we suppose that the radius of convergence of
${\bf z}_{\bf N}(q)$ is strictly larger than ${\bf q}_{\bf N}(1/C)$. 
In that case, ${\bf z}_{\bf N}(q)$ is analytic around ${\bf q}_{\bf
  N}(1/C)$. For the derivative of $\mathbf z_{\mathbf N}$, we have
\begin{equation} \label{eq:zq} 
\mathbf z'_{\mathbf N}\big({\bf q}_{\bf N}(z)\big)=
\frac {1} {{\bf q'}_{\bf N}(z)}.
\end{equation}

Let us first assume that $\Phi>4$.
From Theorem~\ref{theo:2}.$(vi)$, we know that the $m$-th
coefficient of ${\bf q}_{\bf N}(z)$, ${\sf q}_m$ say, 
behaves like $C^m/m^{\Phi/2}$ (up
to a multiplicative constant). Hence, since $\Phi/2>2$,
the series
$
\sum _{m=0} ^{\infty}m{\sf q}_m/C^{m-1}
$
converges, and by Abel's theorem it agrees with the limit 
$$
\om_1:=\underset{\vert z\vert<1/C}{\lim_{z\to 1/C}}
 {{\bf q'}_{\bf N}(z)}.
$$
Hence, by \eqref{eq:zq} and the continuity of ${\bf z}'_{\bf N}(q)$, we have
${\bf z}'_{\bf N}({\bf q}_{\bf N}(1/C))=1/\om_1$, which is different from zero.
As a consequence, by \cite[Theorems~2.4b, 2.4c]{HenrAA}, 
${\bf z}_{\bf N}(q)$ can be inverted in a neighbourhood of 
$q={\bf q}_{\bf N}(1/C)$. This would imply that ${\bf q}_{\bf N}(z)$
is analytic around $1/C$, which contradicts
Theorem~\ref{theo:2}.$(ii)$.

On the other hand, if $\Phi=4$, the above argument has to be adapted
in order to lead to a contradiction. 
To begin with, by Lemma~\ref{lem:exp1/C}.$(iii)$, we have
\begin{equation} \label{eq:singexp2} 
{\bf q}_{\bf N}(z)-
{\bf q}_{\bf N}(1/C)=
\om_2(1-Cz)\log(1-Cz)(1+o(1))
\end{equation}
for $z\to1/C$, where $\om_2$ is some non-zero constant.
Next, we compute the derivative of ${\bf q}_{\bf N}(z)$ using the expression
given in \eqref{eq:G} for the quotient ${\bf G}_{\bf N}(z)/{\bf
  F}_{\bf N}(z)$. Subsequently, we compute its singular expansion for
$z\to1/C$ in the same style as the one for ${\bf q}_{\bf N}(z)$.
The result is that
\begin{equation} \label{eq:singexp3}
{\bf q'}_{\bf N}(z)
=
\om_3\log(1-Cz)(1+o(1))
\end{equation}
for $z\to1/C$, where $\om_3$ is some non-zero constant.
Consequently,
$$
\mathbf z'_{\mathbf N}\big({\bf q}_{\bf N}(1/C)\big)=
\underset{\vert z\vert<1/C}{\lim_{z\to 1/C}}
\frac {1} {{\bf q'}_{\bf N}(z)}=0.
$$
Since $\mathbf z'_{\mathbf N}(q)$ is analytic in a neighbourhood of $q={\bf
  q}_{\bf N}(1/C)$, and since $\mathbf z'_{\mathbf N}(q)$ cannot be
constant (this would imply that $\mathbf q_{\mathbf N}(z)$ is linear,
which contradicts Theorem~\ref{theo:2}.$(ii)$), we have
$$
\mathbf z'_{\mathbf N}(q)=\om_4\big(q-{\bf
  q}_{\bf N}(1/C)\big)^s(1+o(1))
$$
for $q\to{\bf q}_{\bf N}(1/C)$,
where $\om_4$ is a non-zero constant and $s$ is a positive integer.
If we use this in \eqref{eq:zq}, together with \eqref{eq:singexp2} and
\eqref{eq:singexp3}, we obtain
\begin{align*}
\om_4\om_2^s(1-Cz)^s\log^s(1-Cz)(1+o(1))=
\frac {1+o(1)} {\om_3\log(1-Cz)}
\end{align*}
for $z\to1/C$, which is absurd.
\hfill\qed

\begin{Remark} \label{rem:1}
The above argument for the case where $\Phi=4$ does not lead to a
contradiction when applied to the case $\Phi=3$. This is in accordance
with Theorem~\ref{theo:5} and the fact that 
$\exp\big(-\pi\cot(\pi/M)\big)$ is larger than ${\bf q}_{\bf
  N}(1/C_{\bf N})$ in the relevant cases, see Section~\ref{sec:Phi}.
\end{Remark}

\section{Comparison of the two critical values
$\exp(-\pi\cot(\pi/M_{\mathbf N}))$ and ${\bf q}_{\bf N}(1/C_{\bf N})$}
\label{sec:comp}

Because of \eqref{eq:liminfini}, the point $q=\exp(-\pi\cot(\pi/M_{\mathbf
N}))$ is a potential singularity of $\mathbf z_{\mathbf N}(q)$.
In Theorem~\ref{theo:4}.$(ii)$ we have shown that the radius of
convergence of $\mathbf z_{\mathbf N}(q)$ is at most 
${\bf q}_{\bf N}(1/C_{\bf N})$. It is therefore important to know
whether $\exp(-\pi\cot(\pi/M_{\mathbf N}))$ is less\break 
than 
${\bf q}_{\bf N}(1/C_{\bf N})$ or not. In this section, we show that
for $\Phi_{\mathbf N}\ge4$ we have in fact\break
$\exp(-\pi\cot(\pi/M_{\mathbf N}))>{\bf q}_{\bf N}(1/C_{\bf N})$,
which fits well with Conjecture~\ref{conj:1}.

\begin{lem} \label{lem:exp>q(1/C)}
Let $N_1,N_2, \ldots, N_k$ be positive integers, all of which at least $2$,
such that $\Phi=\Phi_{\mathbf N}\ge4$. Furthermore,
let $M=M_{\mathbf N}$, as before. 
Then 
\begin{equation} \label{eq:exp>q{1/C}} 
\exp(-\pi\cot(\pi/M))>\mathbf q_{\mathbf N}(1/C).
\end{equation}
\end{lem}

\begin{proof}
From \cite[inequality on top of p.~157, which, as the proof shows,
  remains valid for {\it real\/} $n$]{SasvAA} we know that for $x\ge1$ we have
\begin{equation} \label{eq:gammaungl}
\Gamma(x+1)=\sqrt{2\pi x}\left(\frac {x} {e}\right)^x e^{\la_x},\quad 
\text{with }\frac {1} {12x+1}<\la_x<\frac {1} {12x}. 
\end{equation}
We use these effective bounds on the gamma function to provide an
upper bound for $\mathbf B_{\mathbf N}(m)$. Let us first suppose that
$m\ge2$. Then we have
\begin{align} 
\notag
\mathbf B_{\mathbf N}(m)&=\prod_{j=1}^k \mathbf{B}_{N_j}(m)
=C^m \prod_{j=1}^k \prod_{i=1}^{\varphi(N_j)} 
\frac{(r_{i,j}/N_j)_m}{m!}
=C^m \prod_{j=1}^k \prod_{i=1}^{\varphi(N_j)} 
\frac{\Ga\left(m+\frac {r_{i,j}} {N_j}\right)}
{\Ga\left(\frac {r_{i,j}} {N_j}\right)\,\Ga(m+1)}\\
&<C^m 
e^{\La_m}
\prod_{j=1}^k \prod_{i=1}^{\varphi(N_j)} 
\frac {1} {\Ga(r_{i,j}/N_j)}
\left(\frac {m+\frac {r_{i,j}} {N_j}-1} {m}\right)^{m+\frac {1} {2}}
\left(\frac {m+\frac {r_{i,j}} {N_j}-1} {e}\right)^{\frac {r_{i,j}}
  {N_j}-1},
\label{eq:U1}
\end{align}
where
$$
\La_m=\sum_{j=1}^k \sum_{i=1}^{\varphi(N_j)}\left(\frac {1} {12(m+\frac
  {r_{i,j}} {N_j}-1)}-
\frac {1} {12m+1}\right).
$$
The quantity $\La_m$ may be estimated from above as follows:
\begin{align} \notag
\La_m&=\sum_{j=1}^k \sum_{i=1}^{\varphi(N_j)}\left(\frac {13-12
\frac {r_{i,j}} {N_j}} 
{12(m+\frac
  {r_{i,j}} {N_j}-1)(12m+1)}\right)
\le\sum_{j=1}^k \sum_{i=1}^{\varphi(N_j)}\left(\frac {13-12
\frac {r_{i,j}} {N_j}} 
{12(m-1)(12m+1)}\right)\\
&\le\sum_{j=1}^k \varphi(N_j)\left(\frac {13-12
\frac {1} {2}} 
{12(m-1)(12m+1)}\right)
\label{eq:U2}
\le \frac {7\Phi} 
{12(m-1)(12m+1)}\le \frac {7} {300}\Phi. 
\end{align}
On the other hand, using the well-known elementary inequality 
$\left(1+\frac {x} {m}\right)^m\le e^x$, valid for $x>-m$, we have
\begin{align} \notag
\left(\frac {m+\frac {r_{i,j}} {N_j}-1} {m}\right)^{m+\frac {1} {2}}
\left(\frac {m+\frac {r_{i,j}} {N_j}-1} {e}\right)^{\frac {r_{i,j}}
  {N_j}-1}&=
\left(1+\frac {\frac {r_{i,j}} {N_j}-1} {m}\right)^{m}
\frac {e^{1-\frac {r_{i,j}}
  {N_j}} \left(1+\frac {\frac {r_{i,j}} {N_j}-1} {m}\right)^{\frac {1} {2}}
}
{\left( {m+\frac {r_{i,j}} {N_j}-1} \right)^{1-\frac {r_{i,j}}
  {N_j}}}\\
&\le
\frac {1
}
{\left( {m-1} \right)^{1-\frac {r_{i,j}}
  {N_j}}}.
\label{eq:U3}
\end{align}
Finally, by the reflection formula (cf.\ \cite[Theorem~1.2.1]{AAR})
$$
\Ga(x)\Ga(1-x)=\frac {\pi} {\sin(\pi x)}
$$
for the gamma function, in the case that all $N_j$'s are different from
$2$ (in which case $\varphi(N_j)$ is always even) we get
\begin{equation} \label{eq:U4}
\prod_{j=1}^k \prod_{i=1}^{\varphi(N_j)} 
\frac {1} {\Ga(r_{i,j}/N_j)}
=
\prod_{j=1}^k \prod_{i=1}^{\varphi(N_j)/2} 
\frac {\sin\left(\pi r_{i,j}/N_j\right)} {\pi}
\le \pi^{-\Phi/2}.
\end{equation}
One can then see that the inequality above holds even if some of the
$N_j$'s should happen to equal $2$.
If \eqref{eq:U2}--\eqref{eq:U4} are 
substituted back in \eqref{eq:U1}, then the result is
\begin{align} \notag
\mathbf B_{\mathbf N}(m)&<C^m 
e^{\frac {7} {300}\Phi}
\pi^{-\Phi/2}
{\left( {m-1} \right)^{-
\sum_{j=1}^k \sum_{i=1}^{\varphi(N_j)} \big(1-\frac {r_{i,j}}
  {N_j}\big)}}\\
\label{eq:U5}
&<
C^m
e^{\frac {7} {300}\Phi}
\pi^{-\Phi/2}
\left( {m-1} \right)^{-\Phi/2},
\end{align}
where we used \eqref{eq:Phi/2} to obtain the last line.

The case of $m=1$ has to be treated separately. In that case, we see
that
\begin{align} 
\mathbf B_{\mathbf N}(1)&=\prod_{j=1}^k \mathbf{B}_{N_j}(1)
=C \prod_{j=1}^k \prod_{i=1}^{\varphi(N_j)} 
\frac{r_{i,j}}{N_j}
=C \prod_{j=1}^k \prod_{i=1}^{\varphi(N_j)/2} 
\frac{r_{i,j}(N-r_{i,j})}{N_j^2}
\le C\cdot 2^{-\Phi}
\label{eq:U6}
\end{align}
if all $N_j$'s are different from $2$. Again, it is readily seen that
the inequality also holds if some of the $N_j$'s should happen to
equal $2$.

We now combine \eqref{eq:U5} and \eqref{eq:U6} in order to estimate
$\mathbf{F}_{\mathbf{N}}(1/C)$:
\begin{align} 
\notag
\mathbf{F}_{\mathbf{N}}(1/C)&=
\sum_{m=0}^{\infty} 
\mathbf{B}_{\mathbf N}(m)
C^{-m}\\
\notag
&< 1+\frac {1} {2^{\Phi}}+
\left(\frac {e^{\frac {7} {300}}}
{\pi^{1/2}}\right)^\Phi
\sum_{m=2}^{\infty}
\left( {m-1} \right)^{-\Phi/2}\\
\notag
&< 1+\frac {1} {2^{\Phi}}+
\left(\frac {e^{\frac {7} {300}}}
{\pi^{1/2}}\right)^\Phi
\zeta(\Phi/2)\\
&< 1+\frac {1} {2^{\max\{4,k\}}}+
\left(\frac {e^{\frac {7} {300}}}
{\pi^{1/2}}\right)^{\max\{4,k\}}
\zeta\big(\max\{4,k\}/2\big).
\label{eq:U7}
\end{align}
In the sequel, we denote the quantity on the right-hand side of
\eqref{eq:U7} by $c_k$.

We are now in the position to establish the inequality
\eqref{eq:exp>q{1/C}} for ``large enough" $\mathbf N$. 
Namely, since
$$
\pi\cot(\pi/M)\le M
$$
and (see \eqref{eq:HlogC})
$$
{\bf G}_{\bf N}(1/C)-\log(C){\bf F}_{\bf N}(1/C)<-\log (C),
$$
it suffices to prove the inequality
\begin{equation*} 
\exp(-M)\ge \exp\left(-\frac {\log C} {\mathbf F_{\mathbf N}(1/C)}\right),
\end{equation*}
or, equivalently,
\begin{equation*}
\log C\ge M\cdot \mathbf F_{\mathbf N}(1/C).
\end{equation*}
We now make use of the estimation (cf.\ \cite[Theorem~8.8.7]{BaShAA})
\begin{equation} \label{eq:phiappr} 
\varphi(n)\ge\max\left\{1,\frac {n} {e^\ga\log\log n+\frac {3} {\log \log n}}\right\}
\end{equation}
for the totient function, where $\ga$ denotes Euler's constant. 
For convenience, let us write
$\overline\varphi(n)$ for the right-hand side of \eqref{eq:phiappr}.
Then, use of \eqref{eq:phiappr} in the definition of $C$ gives
$$
\log C\ge 
\sum _{j=1} ^{k}\log N_j^{\varphi(N_j)}=
\sum _{j=1} ^{k}\varphi(N_j)\log N_j\ge
\sum _{j=1} ^{k}\overline\varphi(N_j)\log N_j.
$$
If we put this
together with \eqref{eq:U7}, we see that the inequality
\eqref{eq:exp>q{1/C}} will be proved whenever
\begin{equation} \label{eq:Krit} 
\sum _{j=1} ^{k}\overline\varphi(N_j)\log N_j
\ge c_kM.
\end{equation}
It remains to consider the cases where \eqref{eq:Krit} does not hold.
We claim that this is only a finite number of cases. Indeed,
for fixed $k$, there can only be a finite number of $k$-tuples
$(N_1,N_2,\dots,N_k)$ for which \eqref{eq:Krit} is violated, since 
trivially $c_k$ is constant for fixed $k$,
and since $\overline\varphi(M)\log M$ grows faster than $M$. 
On the other hand, if $k\ge15$, then \eqref{eq:Krit} holds
automatically. For, we have
$$
\min_{M\ge2}\{\overline\varphi(M)\log M-c_{15}M\}\ge-9,
$$
and therefore (assuming that $M=N_1$ without loss of generality)
\begin{align*} 
\sum _{j=1} ^{k}\overline\varphi(N_j)\log N_j&=
\overline\varphi(M)\log M+
\sum _{j=2} ^{k}\overline\varphi(N_j)\log N_j\\
&\ge c_{15}M-9+(k-1)\log 2\\
&\ge c_{k}M-9+(k-1)\log 2\\
&\ge c_kM
\end{align*}
for $k\ge15$, where we used the simple fact that the $c_k$'s are
monotone decreasing in $k$ in the third line.

In summary, there is indeed only a finite number of cases left where
\eqref{eq:Krit} does not hold. For these cases, we have verified on
the computer that the inequality \eqref{eq:exp>q{1/C}} is satisfied.
For carrying out this verification, we used the estimation
\begin{align} \notag
\mathbf q_{\mathbf N}(1/C)
&=\exp\left(\frac {\sum_{m=0}^{\infty} 
\mathbf{B}_{\mathbf N}(m)
\big(\mathbf{H}_N(m)-\log C\big)C^{-m}} {
\sum_{m=0}^{\infty} 
\mathbf{B}_{\mathbf N}(m)
C^{-m}}\right)\\
\notag
&<\exp\left(\frac {\sum_{m=0}^{20} 
\mathbf{B}_{\mathbf N}(m)
\big(\mathbf{H}_N(m)-\log C\big)C^{-m}} {
\sum_{m=0}^{20} 
\mathbf{B}_{\mathbf N}(m)
C^{-m}+
e^{\frac {7} {300}\Phi}
\pi^{-\Phi/2}
\sum _{m=21} ^{\infty}\left( {m-1} \right)^{-\Phi/2}}\right)\\
&<\exp\left(\frac {\sum_{m=0}^{20} 
\mathbf{B}_{\mathbf N}(m)
\big(\mathbf{H}_N(m)-\log C\big)C^{-m}} {
\sum_{m=0}^{20} 
\mathbf{B}_{\mathbf N}(m)
C^{-m}+
e^{\frac {7} {300}\Phi}
\pi^{-\Phi/2}
\left(\frac {\Phi} {2}-1\right)^{-1}19^{-\frac {\Phi} {2}+1}}
\right)
\label{eq:U8}
\end{align}
(with the third line following again from \eqref{eq:HlogC}
and \eqref{eq:U5}, and the
last line from\break 
$\sum _{m=21} ^{\infty}\left( {m-1} \right)^{-\Phi/2}<
\int _{19} ^{\infty}x^{-\Phi/2}\,\dd x$), and we actually
compared the left-hand side of \eqref{eq:exp>q{1/C}} with the upper
bound on the right-hand side given in \eqref{eq:U8}. This completes
the proof of the lemma.
\end{proof}

\section{Proof of Theorem~\ref{theo:5}: the cases $\Phi_{\mathbf N}=1,2,3$}
\label{sec:Phi}

The case of $\Phi=1$ is a trivial case. We provide the details
nevertheless for the sake of completeness. For the cases $\Phi=2,3$,
it is well documented in the literature that 
the corresponding functions $\mathbf z_{\mathbf N}(q)$ live in the world of
modular forms. We refer the reader to \cite{StillAA,zudram} for reviews of the
corresponding classical theory.
Despite of this, it seems that the questions that
we treat in the present paper --- namely questions concerning the
analytic nature of the function $\mathbf z_{\mathbf N}(q)$ --- 
have not been recorded in sufficient
detail to extract complete information about the singularities of the
function, for example, 
and, thus, of the radius of convergence of the Taylor series at
$q=0$. 

The purpose of this section is to compile the relevant facts, to provide
a coherent overview, and to fill possible gaps whenever necessary.

Since we shall make use of it below, we record here some basic facts
on modular forms. Given a subgroup $\Ga$ of the modular group
$SL_2(\mathbb Z)$ and a non-negative integer $k$, 
a function $f(\tau)$ from the complex upper half plane
$\mathbb H$ to the complex numbers is called a {\it modular form
of weight $k$ for
$\Ga$} if it is meromorphic on $\mathbb H$ and
\begin{equation} \label{eq:action}
f(T(\tau))=(c\tau+d)^k f(\tau) 
\end{equation}
for all elements $T=\left(\begin{smallmatrix}
    a&b\\c&d\end{smallmatrix}\right)$ of $\Ga$, where the action of
$T$ is defined by $T(\tau)=(a\tau+b)/(c\tau+d)$. Modular forms of
weight $k$ for $\Ga$ satisfy the {\it valence formula}
(cf.\ \cite[Theorem~4.1.4]{RankAA})
\begin{equation} \label{eq:valform}
\sum _{\zeta\in \mathbb F} ^{}\frac {\ord(f,\zeta)} 
{\vert \stab_\zeta(\widehat\Ga)\vert} = \frac {\vert\widehat{SL_2(\mathbb
    Z)}/\widehat \Ga\vert } {12}k.
\end{equation}
Here, 
\begin{enumerate}
\item[$-$]
$\mathbb F$ is a fundamental region for the action of $\Ga$ on
$\mathbb H$ to which one adds the cusps (a system of representatives
of the orbit of $\infty$ under the action of $SL_2(\mathbb Z)$ when
restricted to $\Ga$);
\item[$-$]
 $\widehat\Ga$ is the {\it mapping group} corresponding to
$\Ga$, i.e., the group arising from $\Ga$ upon identification of $T$
and $-T$ if both of them should be contained in $\Ga$ (the action
\eqref{eq:action} on $\mathbb H$ of $T$ and $-T$ is identical); 
\item[$-$]
$\ord(f,\zeta)$ is the usual order of $f$ at $\zeta$ if
$\zeta$ is not a cusp;
\item[$-$]
if $\zeta$ is a cusp, $\zeta=T\infty$ for $T=\left(\begin{smallmatrix}
    a&b\\c&d\end{smallmatrix}\right)\in SL_2(\mathbb Z)$, then
$\ord(f,\zeta)$ is defined as the order of the series
expansion of $(c\tau +d)^{-k}f(T(\tau))$ in $\widetilde q=\exp(2i\pi \tau/n_T)$ 
(as a Laurent series in $\widetilde q$), where $k$ is the weight of $f$, and
where $n_T$ is the least positive integer such that $\left(\begin{smallmatrix}
    1&1\\0&1\end{smallmatrix}\right)^{n_T}\in T\Ga T^{-1}$; 
\item[$-$]
$\stab_\zeta(\widehat\Ga)$ is the subgroup of $\widehat\Ga$ consisting of the
elements fixing $\zeta$.
\end{enumerate}

The stabiliser $\stab_i{\widehat\Ga}$ of $i$ in $\widehat\Ga$ can be
$2$ or $1$ depending on whether or not $V=\left(\begin{smallmatrix}
    0&-1\\1&0\end{smallmatrix}\right)$ or $-V$ is in $\Ga$.
The stabiliser $\stab_\rho{\widehat\Ga}$ of $\rho=\exp(2i\pi /3)$ 
in $\widehat\Ga$ can be
$3$ or $1$ depending on whether or not $P=\left(\begin{smallmatrix}
    0&-1\\1&1\end{smallmatrix}\right)$ or $-P$ is in $\Ga$.
All other stabilisers consist only of one element.

Below we shall make frequent use of the following special functions.
The {\it Eisenstein series} (in Ramanujan's notation) 
$Q(q)$ and $R(q)$ are defined by
$$
Q(q)= 1+240
\sum _{n=1} ^{\infty}\si_3(n)q^n\quad \quad \text{and}\quad \quad 
R(q)= 1-504
\sum _{n=1} ^{\infty}\si_5(n)q^n
$$
and $\si_k(n)=
\sum _{d\mid n} ^{}d^k$. 
The Eisenstein series $E_4(\tau):=Q\big(\exp(2i\pi \tau)\big)$ 
is a modular form of weight $4$ for the full
modular group $SL_2(\mathbb Z)$, while $E_6(\tau):=R\big(\exp(2i\pi \tau)\big)$ 
is a modular form of weight $6$ for the same group.
It is well-known (see \cite[p.~143]{SerrAA}) that the zeroes of 
$Q(q)$ are of the form
$q=\exp(2i\pi \tau)$, where $\tau$ runs through the elements of the
orbit of $\rho$ under $SL_2(\mathbb Z)$. All of these are simple
zeroes. Similarly (see \cite[p.~143]{SerrAA}), 
$R(q)$ has only simple zeroes which are of the form
$q=\exp(2i\pi \tau)$, where $\tau$ runs through the elements of the
orbit of $i$ under $SL_2(\mathbb Z)$. The unit circle is a natural
boundary for both $Q(q)$ and $R(q)$.

The {\it Dedekind--Klein $j$-invariant} is given by 
$$
j(\tau)=1728\frac {E_4^3(\tau)} {E_4^3(\tau)-E_6^2(\tau)}.
$$
We shall rather use the variant
\begin{equation} \label{eq:Jdef} 
J(q)=\frac {Q^3(q)} {Q^3(q)-R^2(q)},
\end{equation}
so that $j(\tau)=1728 J(\exp(2i\pi\tau))$.
The function $j(\tau)$ is a modular form of
weight $0$ for $SL_2(\mathbb Z)$. $J(q)$ is meromorphic in the unit disk 
with a unique pole at $q=0$, which is a simple pole, and the unit
circle is a natural boundary.

The {\it Dedekind $\eta$-function} is defined by 
$$
\eta(\tau)=\exp(i\pi\ta/12)
\prod _{n=1} ^{\infty}(1-\exp(2i\pi n\ta)).
$$
We shall rather use the variant
$$
H(q)=q^{1/24}
\prod _{n=1} ^{\infty}(1-q^n).
$$
The function $\eta^{24}(\tau)$ can be expressed in
terms of Eisenstein series in the form
$$
\eta^{24}(\tau)=E_4^3(\tau)-
E_6^2(\tau),
$$
and it is therefore a modular form of
weight $12$ for $SL_2(\mathbb Z)$.

Two of Jacobi's theta functions will sometimes appear: the function
$\th_2(q)$ defined by
$$
\th_2(q)=
\sum _{n=-\infty} ^{\infty}q^{(n+\frac {1} {2})^2}=
2q^{1/4}
\prod _{j=1} ^{\infty}(1-q^{2j})(1+q^{2j})^2
$$
(the equality of the two expressions above follows from Jacobi's
triple product identity, see \cite[Theorem~10.4.1]{AAR}), and the
function 
$\th_3(q)$ defined by
$$
\th_3(q)=
\sum _{n=-\infty} ^{\infty}q^{n^2}=
\prod _{j=1} ^{\infty}(1-q^{2j})(1+q^{2j-1})^2
$$
(with again Jacobi's triple product identity explaining the equality).

We shall also frequently use the fact that, if $f(\tau)$ is a modular
form of weight $k$ for $SL_2(\mathbb Z)$, then $f(N\tau)$ is a modular
form of weight $k$ for $\Ga_0(N)$, where $\Ga_0(N)$ is the subgroup of
$SL_2(\mathbb Z)$ consisting of all matrices $\left(\begin{smallmatrix}
    a&b\\c&d\end{smallmatrix}\right)$ with $c\equiv0$~mod~$N$.
Furthermore, we have $\vert\widehat{SL_2(\mathbb
    Z)}/\widehat \Ga_0(2)\vert=3$ and 
$\vert\widehat{SL_2(\mathbb
    Z)}/\widehat \Ga_0(3)\vert=4$ (see \cite[(1.4.23)]{RankAA}).

\medskip
We are now ready to discuss all cases in which $\Phi=1,2,3$.

\subsection{$\Phi=1$}
There is only a single example in which $\Phi=1$, namely if $\mathbf
N=(2)$. In that case (see \eqref{eq:Phi=1}), we
have $\mathbf q_{(2)}(z)=(1-\sqrt{1-4z})^2/(4z)$ and $\mathbf
z_{(2)}(q)=q/(1+q)^2$. Hence, 
$\mathbf q_{(2)}(z)$ has exactly one singularity at $z=1/4$,
which is of square
root type, and the radius of convergence is $1/4=1/C$. Moreover,
$\mathbf z_{(2)}(q)=q/(1+q)^2$ has exactly one singularity at
$q=-1$. Its radius of
convergence is $1=\exp(-\pi\cot(\pi/2))$.

\subsection{$\Phi=2$}
\label{sec:Phi=2}
In this case, according to Theorem~\ref{theo:2}.$(ii)$ and 
the remark after the statement of
the theorem, the point $z=1/C$ is always a singularity of
$\mathbf q_{\mathbf N}(z)$, with $\mathbf q_{\mathbf N}(1/C)=1$.\break
Moreover, because of \eqref{eq:1/C2}, this singularity is of 
``$1/\log$''-type. Consequently, by Lemma~\ref{lem:qzess}, 
the mirror map $\mathbf z_{\mathbf
  N}(q)$ has an essential singularity at $q=\mathbf q_{\mathbf
N}(1/C)=1$.

\smallskip
{\sc Case $\mathbf N=(3)$}. Here, we have
$$
\mathbf F_{(3)}(z)=
\sum _{n=0} ^{\infty}\frac {(3n)!} {n!^3}z^n=
{}_{2}F_1
 \!\left [ \begin{matrix} {1/3,2/3}\\ {
   1}\end{matrix} ; {\displaystyle 27z}\right ].
$$
It is well-known (cf.\ e.g.\ \cite[(23)]{zudram}) that
\begin{equation} \label{eq:z3} 
\mathbf z_{(3)}(q)=(Q_3^{3/2}(q)-R_3(q))/(54Q_3^{3/2}(q)),
\end{equation}
 where
$Q_3(q)=(Q(q)+9Q(q^3))/10$ and $R_3(q)=(R(q)+27R(q^3))/28$.
An alternative expression (cf.\ \cite[Theorem~2.3]{BoBoAA}) is given by
\begin{align*} 
\mathbf  z_{(3)}(q)&=
\frac {\left(\th_2(q^{1/3})\th_2(q)+\th_3(q^{1/3})\th_3(q)
-\th_2(q)\th_2(q^3)-\th_3(q)\th_3(q^3)\right)^3}
{216\left(\th_2(q)\th_2(q^3)+\th_3(q)\th_3(q^3)\right)^3}\\
&=
\frac {q} {\displaystyle 27q+
\prod _{n=1} ^{\infty}\frac {(1-q^n)^{12}} {(1-q^{3n})^{12}}}=
\frac {1} {27+\dfrac {H^{12}(q)} {H^{12}(q^3)}}.
\end{align*}
The second equality can be easily derived
using standard reasoning that involves an estimation of the orders of
both expressions at the cusps of a fundamental region of $\Ga_0(12)$,
application of the valence formula \eqref{eq:valform}, and a
verification that sufficiently many coefficients in the power series
expansion in $q$ agree; see \cite[Sections~3 and 4]{GarvAA} for a
detailed description of this kind of argument.

Again alternatively, $\mathbf z_{(3)}(q)$ is
solution of the equation (cf.\ \cite[Sec.~5.1, p.~176]{lianyau1})
\begin{equation} \label{eq:Jz3}
1728J(q)=\frac {(1+216\mathbf z_{(3)}(q))^3} {\mathbf
  z_{(3)}(q)(1-27\mathbf z_{(3)}(q))^3}.
\end{equation}
By Lemma~\ref{lem:expinf} (with $M=3$, $L=3$, $R=2/3$)
and Lemma~\ref{lem:qz}, we know that $\mathbf z_{(3)}(q)$
has a pole of order $3$ at $q=-\exp(-\pi/\sqrt 3)$. (This corrects
the statement in the paragraph containing (3) in \cite{zudcast}.)
Hence, by \eqref{eq:z3},
$q=-\exp(-\pi/\sqrt 3)$ must be a zero of $Q_3(q)$ of order at least
$2$. 
Since the function $E_4(\tau)$ is a modular form
of weight $4$ for $SL_2(\mathbb Z)$, the function $\widetilde
Q_3(\tau):=Q_3\big(\exp(2i\pi
\tau)\big)$ is a modular form of weight $4$ for $\Ga_0(3)$.
It is easy to see that the orbit of $\rho$ under $SL_2(\mathbb Z)$
splits under the action of $\Ga_0(3)$ into three orbits: one orbit
containing $\rho$, another orbit containing $\rho+1=\frac {1}
{2}+\frac {i\sqrt3} {2}$, and another orbit containing 
$\rho_1=\tfrac {1} {2}+\tfrac {i} {2\sqrt3}$. The reader should note
that $\exp(2i\pi\rho_1)=-\exp(-\pi/\sqrt3)$.
If we now apply the valence formula \eqref{eq:valform} to $\widetilde 
Q_3(\tau)$, then, since $\vert \stab_{\rho_1}(\widehat\Ga_0(3))\vert=3$,
while $\vert \stab_{\rho}(\widehat\Ga_0(3))\vert=1$
and $\vert \stab_{i}(\widehat\Ga_0(3))\vert=1$, we obtain
$$\frac {1} {3}\ord\big(\widetilde Q_3,\rho_1\big)+
\sum _{\zeta\in \mathbb F\backslash \{\rho_1\}} ^{}{\ord(\widetilde Q_3,\zeta)} 
=\frac {4} {3}.$$
We know that $\widetilde Q_3(\tau)$ 
is analytic in the upper half plane, that
$\ord(\widetilde Q_3,\zeta)=0$ for all cusps $\zeta$, and that
$\ord(\widetilde Q_3,\rho_1)\ge2$. 
Therefore, if the above equation wants to hold then
necessarily $\ord(\widetilde Q_3,\rho_1)=4$, and all other orders must be zero.
(As a by-product, we see that $R_3(q)$ has a zero of order $3$ at
$q=-\exp(-\pi/\sqrt3)$. We remark that the valence formula applied to
$R_3(q)$ implies that there exists another, simple zero in a fundamental
region of $\widehat\Ga_0(3)$. We claim that this is
$q=\exp(-2\pi\sqrt{3})$ and all elements in its orbit. To see this,
one first observes that \eqref{eq:Jdef} and the previously quoted fact
that $Q\big(-\exp(-\pi\sqrt{3})\big)=0$ imply that
$J\big(-\exp(-\pi\sqrt{3})\big)=0$. Combining this with
the modular equation of level $2$, which provides a polynomial
relation between $J(q)$ and $J(q^2)$ (cf.\ \cite[expression for
  $H_2(x,y)$ given on p.~192]{HerrAA}), we
infer that
$J\big(\exp(-2\pi\sqrt{3})\big)=125/4$. If this is substituted in
\eqref{eq:Jz3}, one obtains four possible values for
$z_{(3)}\big(\exp(-2\pi\sqrt{3})\big)$, one of which is
$1/54$. By a numerical calculation, the other three values can be
ruled out. Finally, by inserting 
$z_{(3)}\big(\exp(-2\pi\sqrt{3})\big)=1/54$ in \eqref{eq:z3},
we see that $R\big(\exp(-2\pi\sqrt{3})\big)=0$, as we claimed.)

In summary, the above arguments prove that $\mathbf z_{(3)}(q)$ is
meromorphic on the unit disk $\vert q\vert<1$ with poles of third
order at all points $\exp(2i\pi \tau)$, where $\tau$ is an
element of the orbit of $\rho_1$ under $\Ga_0(3)$.
It is not difficult to see that $\exp(2i\pi\rho_1)=-\exp(-\pi/\sqrt3)$ 
is the element of smallest
modulus in this set, and that the other elements become dense near
the boundary $\vert q\vert=1$. 
In particular, the Taylor
expansion of $\mathbf z_{(3)}(q)$ has radius of convergence 
$\exp(-\pi/\sqrt 3)=\exp(-\pi\cot(\pi/3))=0.163033\dots$,
and $\mathbf z_{(3)}(q)$ has $\vert q\vert=1$ as natural boundary.

\smallskip
{\sc Case $\mathbf N=(4)$}. Here, we have
$$
\mathbf F_{(4)}(z)=
\sum _{n=0} ^{\infty}\frac {(4n)!} {(2n)!\,n!^2}z^n=
{}_{2}F_1
 \!\left [ \begin{matrix} {1/4,3/4}\\ {
   1}\end{matrix} ; {\displaystyle 64z}\right ].
$$
It is known (cf.\ \cite[Theorem~2.6]{BoBoAA}) that
\begin{equation} \label{eq:z4} 
 \mathbf z_{(4)}(q)=\frac {\th_2^4(q)\th_3^4(q)} 
{16\,(\th_2^4(q)+\th_3^4(q))^2}=
\frac {q} {\displaystyle 64q+
\prod _{n=1} ^{\infty}\frac {(1-q^n)^{24}} {(1-q^{2n})^{24}}}=
\frac {1} {64+\dfrac {H^{24}(q)} {H^{24}(q^2)}}.
\end{equation}
(The second equality can again be proven by arguments as in
\cite[Sections~3 and 4]{GarvAA}.)

Again alternatively, it is well-known that
$\mathbf z_{(4)}(q)$ is
solution of the equation (cf.\ \cite[Sec.~5.1, p.~176]{lianyau1})
$$
1728J(q)=\frac {(1+192\mathbf z_{(4)}(q))^3} {\mathbf
  z_{(4)}(q)(1-64\mathbf z_{(4)}(q))^2}.
$$
By Lemma~\ref{lem:expinf} (with $M=4$, $L=4$, $R=3/4$)
and Lemma~\ref{lem:qz}, we know that $\mathbf z_{(4)}(q)$
has a pole of order $2$ at $q=-\exp(-\pi)$. 
Hence, by \eqref{eq:z4},
$q=-\exp(-\pi)$ must be a zero of $64+\frac {H^{24}(q)}
{H^{24}(q^2)}$ of order $2$. 
The function $f(\tau):=64+\frac {H^{24}(\exp(2i\pi \tau))}
{H^{24}(\exp(4i\pi \tau))}$ is a modular form
of weight $0$ for $\Ga_0(2)$.
It is easy to see that the orbit of $i$ under $SL_2(\mathbb Z)$
splits under the action of $\Ga_0(2)$ into two orbits: one orbit
containing $i$ and another orbit containing 
$\de_1=-\tfrac {1} {2}+\tfrac {i} {2}$. The reader should note
that $\exp(2i\pi\de_1)=-\exp(-\pi)$.
If we now apply the valence formula \eqref{eq:valform} to $f(\tau)$, 
then, since $\vert \stab_{\de_1}(\widehat\Ga_0(2))\vert=2$,
while $\vert \stab_{i}(\widehat\Ga_0(2))\vert=1$,
$\ord_\infty(f)=-1$,
$\ord_{V\infty}(f)=\ord_0(f)=0$,
and $\vert \stab_{i}(\widehat\Ga_0(2))\vert=1$, we obtain
$$-1+\frac {1} {2}\ord\big(f,\de_1\big)+
\sum _{\zeta\in \mathbb F\backslash \{\de_1,\infty,V\infty\}} ^{}{\ord(f,\zeta)} 
=0.$$
We know that $f$ is analytic in the upper half plane and that
$\ord(f,\de_1)=2$. 
Therefore, if the above equation wants to hold then
necessarily all other orders must be zero.

In summary, the above arguments prove that $\mathbf z_{(4)}(q)$ is
meromorphic on the unit disk $\vert q\vert<1$ with poles of second
order at all points $\exp(2i\pi \tau)$, where $\tau$ is an
element of the orbit of $\de_1$ under $\Ga_0(2)$.
It is not difficult to prove that $\exp(2i\pi\de_1)=-\exp(-\pi)$ 
is the element of smallest
modulus in this set, and that the other elements become dense near
the boundary $\vert q\vert=1$. 
In particular, the Taylor
expansion of $\mathbf z_{(4)}(q)$ has radius of convergence 
$\exp(-\pi)=\exp(-\pi\cot(\pi/4))=0.0432139\dots$,
and $\mathbf z_{(4)}(q)$ has $\vert q\vert=1$ as natural boundary.

\smallskip
{\sc Case $\mathbf N=(6)$}. Here, we have
$$
\mathbf F_{(6)}(z)=
\sum _{n=0} ^{\infty}\frac {(6n)!} {(3n)!\,(2n)!\,n!}z^n=
{}_{2}F_1
 \!\left [ \begin{matrix} {1/6,5/6}\\ {
   1}\end{matrix} ; {\displaystyle 432z}\right ].
$$
It is well-known that $\mathbf z_{(6)}(q)$ is
solution of the equation (cf.\ \cite[Sec.~5.1, p.~176]{lianyau1})
\begin{equation} \label{eq:z6} 
1728J(q)=\frac {1} {\mathbf
  z_{(6)}(q)(1-432\mathbf z_{(6)}(q))}.
\end{equation}
By Lemma~\ref{lem:expinf} (with $M=6$, $L=6$, $R=5/6$)
and Lemma~\ref{lem:qz}, we know that $\mathbf z_{(6)}(q)$
has an algebraic branch point with exponent $-3/2$ at $q=-\exp(-\pi\sqrt 3)$. 
More precisely, since, by solving \eqref{eq:z6}, we have
$$\mathbf z_{(6)}(q)=\frac {1} {864}\left(1-\sqrt{\frac {J(q)-1} 
{J(q)}}\right),$$
the only singularities of $\mathbf z_{(6)}(q)$ in the interior of
the unit disk can occur at points where
$J(q)=0$ or $J(q)=1$. These are points $q$ where $Q(q)=0$,
respectively where $R(q)=0$. The
corresponding values of $q$ are $q=\exp(2i\pi\tau)$, where $\tau$ runs
through the elements of the orbit of $\rho$ under $SL_2(\mathbb Z)$,
respectively the elements of the orbit of $i$ under $SL_2(\mathbb Z)$.
However, since $R(q)$ appears as a square in the
definition of $J(q)$, each point $q=\exp(2i\pi\tau)$, where $\tau$ is
in the orbit of $i$, is a zero of even
order of $J(q)-1$, whence $\sqrt{J(q)-1}$ is analytic at these
points. 

In summary, 
the singularities of $\mathbf z_{(6)}(q)$ in the interior of the
unit disk are $q=\exp(2i\pi\tau)$, where $\tau$ runs through the
elements of the orbit of $\rho$ under $SL_2(\mathbb Z)$. 
Each of them is a branch point.
It is not difficult to prove that $\exp(2i\pi\rho)=-\exp(-\pi\sqrt3)$ 
is the element of smallest
modulus in this set, and that the other elements become dense near
the boundary $\vert q\vert=1$. 
In particular, the Taylor
expansion of $\mathbf z_{(6)}(q)$ has radius of convergence 
$\exp(-\pi\sqrt 3)=\exp(-\pi\cot(\pi/6))=0.00433342\dots$,
and $\mathbf z_{(6)}(q)$ has $\vert q\vert=1$ as natural boundary.

\smallskip
{\sc Case $\mathbf N=(2,2)$}. Here, we have
$$
\mathbf F_{(2,2)}(z)=
\sum _{n=0} ^{\infty}\frac {(2n)!^2} {n!^4}z^n=
{}_{2}F_1
 \!\left [ \begin{matrix} {1/2,1/2}\\ {
   1}\end{matrix} ; {\displaystyle 16z}\right ].
$$
It is well-known (cf.\ \cite[Theorem~2.2]{BoBoAA}) that 
\begin{equation} \label{eq:z22} 
\mathbf z_{(2,2)}(q)=\frac {\th_2^4(q)} {16\,\th_3^4(q)}=q
\prod _{n=1} ^{\infty}\frac {(1-q^{4n})^8} {(1-(-q)^n)^8}
=e^{i\pi/3}\frac {H^8(q^4)} {H^8(-q)}.
\end{equation}
(In this case, the second equality results upon minor simplification
from the product expressions for $\th_2(q)$ and $\th_3(q)$.)

Alternatively, $\mathbf z_{(2,2)}(q)$ is
solution of the equation (cf.\ \cite[Sec.~5.1, p.~176]{lianyau1})
\begin{equation} \label{eq:z2,2} 
1728J(q)=\frac {(1+224z_{(2,2)}(q)+256z_{(2,2)}^2(q))^3} {\mathbf
  z_{(2,2)}(q)(1-16\mathbf z_{(2,2)}(q))^4}.
\end{equation}
From \eqref{eq:z22}, it is evident that the radius of convergence of
$\mathbf z_{(2,2)}(q)$ is $1=\exp(-\pi\cot(\pi/2))$, and 
that $\mathbf z_{(2,2)}(q)$ has $\vert q\vert=1$ as natural boundary.
We remark that this is consistent with Lemma~\ref{lem:expinf}.$(iii)$
and Lemma~\ref{lem:qz}.$(iii)$, which imply that $\mathbf
  z_{(2,2)}(q)$ has an essential singularity at
  $q=-\exp(-\pi\cot(\pi/2))=-1$.  

\subsection{$\Phi=3$}
\label{sec:Phi=3}
In this case, according to Theorem~\ref{theo:2}.$(ii)$,
the point $z=1/C$ is always a singularity of
$\mathbf q_{\mathbf N}(z)$. Furthermore, because of 
Lemma~\ref{lem:exp1/C}.$(ii)$ (with $d=0$), this singularity is of
square root type.
Consequently, according to Lemma~\ref{lem:qz2} (with $d=0$),
the mirror map $\mathbf z_{\mathbf N}(q)$ is analytic at $\mathbf
q_{\mathbf N}(1/C)$, with $\mathbf z'_{\mathbf N}(\mathbf
q_{\mathbf N}(1/C))=0.$

\smallskip
{\sc Case $\mathbf N=(2,3)$}. Here, we have
$$
\mathbf F_{(2,3)}(z)=
\sum _{n=0} ^{\infty}\frac {(2n)!\,(3n)!} {n!^5}z^n=
{}_{3}F_2
 \!\left [ \begin{matrix} {1/2,1/3,2/3}\\ {
  1, 1}\end{matrix} ; {\displaystyle 108z}\right ]=
{}_{2}F_1
 \!\left [ \begin{matrix} {1/3,1/6}\\ {
   1}\end{matrix} ; {\displaystyle 108z}\right ]^2.
$$
The identity between the hypergeometric series is a special case of
Clausen's formula (cf.\ \cite[(2.5.7)]{SlatAC})
\begin{equation} \label{eq:Clausen} 
{}_{3}F_2
 \!\left [ \begin{matrix} {2a,2b,a+b}\\ {
   2a+2b, a+b+\frac {1} {2}}\end{matrix} ; {\displaystyle z}\right ]
=
{}_{2}F_1
 \!\left [ \begin{matrix} {a,b}\\ {
   a+b+\frac {1} {2}}\end{matrix} ; {\displaystyle z}\right ]^2.
\end{equation}
It is well-known (cf.\ e.g.\ \cite[(23)]{zudram}) that
$\mathbf z_{(2,3)}(q)=(Q_3^{3}(q)-R^2_3(q))/(108Q_3^{3}(q))$, where
$Q_3(q)$ and $R_3(q)$ are defined as in the case $\mathbf N=(3)$.
From the considerations concerning the zeroes of 
$Q_3(q)$ and $R_3(q)$ in the case $\mathbf N=(3)$,
we conclude that $\mathbf z_{(2,3)}(q)$
has poles of order $6$ at all points $\exp(2i\pi \tau)$, where $\tau$ is an
element of the orbit of $\rho_1=\tfrac {1} {2}+\tfrac {i} {2\sqrt3}$ 
under $\Ga_0(3)$. The point
$\exp(2i\pi\rho_1)=-\exp(-\pi/\sqrt3)$ 
is the element of smallest
modulus in this set, and the other elements become dense near
the boundary $\vert q\vert=1$. 
In particular, the Taylor
expansion of $\mathbf z_{(2,3)}(q)$ has radius of convergence 
$\exp(-\pi/\sqrt 3)=\exp(-\pi\cot(\pi/3))=0.163033\dots$,
and $\mathbf z_{(2,3)}(q)$ has $\vert q\vert=1$ as natural boundary.

{\sc Case $\mathbf N=(2,4)$}. Here, we have
$$
\mathbf F_{(2,4)}(z)=
\sum _{n=0} ^{\infty}\frac {(4n)!} {n!^4}z^n=
{}_{3}F_2
 \!\left [ \begin{matrix} {1/2,1/4,3/4}\\ {
  1, 1}\end{matrix} ; {\displaystyle 256z}\right ]=
{}_{2}F_1
 \!\left [ \begin{matrix} {1/8,3/8}\\ {
   1}\end{matrix} ; {\displaystyle 256z}\right ]^2,
$$
again with Clausen's formula \eqref{eq:Clausen} explaining the identity between the
hypergeometric series.
It is well-known (cf.\ e.g.\ \cite[(23)]{zudram}) that
\begin{equation} \label{eq:z24} 
\mathbf z_{(2,4)}(q)=(Q_2^{3}(q)-R^2_2(q))/(256Q_2^{3}(q)),
\end{equation}
where
$Q_2(q)=(Q(q)+4Q(q^2))/5$ and $R_2(q)=(R(q)+8R(q^2))/9$.
An alternative expression can be found, if one observes that, by the
quadratic transformation formula (see \cite[Ex.~4.(iii), p.~97]{BailAA})
\begin{equation} \label{eq:trans} 
{} _{2} F _{1} \!\left [ \begin{matrix} { a, b}\\ { {\frac1 2} + a + b}\end{matrix} ;
   {\displaystyle z}\right ] =
  {} _{2} F _{1} \!\left [ \begin{matrix} { 2 a, 2 b}\\ { {\frac {1} {2}} + a +
   b}\end{matrix} ; {\displaystyle {\frac {1 - {\sqrt{1 - z}}} 2}}\right ],
\end{equation}
we have
$$
{}_{2}F_1
 \!\left [ \begin{matrix} {1/8,3/8}\\ {
   1}\end{matrix} ; {\displaystyle 256z}\right ]=
  {} _{2} F _{1} \!\left [ \begin{matrix} { 1/4, 3/4}\\ {1
   }\end{matrix} ; {\displaystyle {\frac {1 - {\sqrt{1 - 256z}}} 2}}\right ],
$$
which enables us to identify $\frac {1} {2}\big({ {1 - {\sqrt{1 - 256\mathbf
        z_{(2,4)}(q)}}} }\big)$ with $64\mathbf z_{(4)}(q)$, or,
explicitly, 
\begin{equation} \label{eq:z24/z4} 
\mathbf z_{(2,4)}(q)=\mathbf z_{(4)}(q)-64\mathbf z_{(4)}^2(q).
\end{equation}
By Lemma~\ref{lem:expinf} (with $M=4$, $L=4$, $R=1/2$)
and Lemma~\ref{lem:qz}, we know that $\mathbf z_{(2,4)}(q)$
has a pole of order $4$ at $q=-\exp(-\pi)$. 
We can now either use the expression \eqref{eq:z24} in combination
with arguments as in the case $\mathbf N=(4)$, or \eqref{eq:z24/z4}
and similar arguments, to conclude that
$\mathbf z_{(2,4)}(q)$ is
meromorphic on the unit disk $\vert q\vert<1$ with poles of fourth
order at all points $\exp(2i\pi \tau)$, where $\tau$ is an
element of the orbit of $\de_1=-\tfrac {1} {2}+\tfrac {i} {2}$ 
under $\Ga_0(2)$.
The point $\exp(2i\pi\de_1)=-\exp(-\pi)$ 
is the element of smallest
modulus in this set, and the other elements become dense near
the boundary $\vert q\vert=1$. 
In particular, the Taylor
expansion of $\mathbf z_{(2,4)}(q)$ has radius of convergence 
$\exp(-\pi)=\exp(-\pi\cot(\pi/4))=0.0432139\dots$,
and $\mathbf z_{(2,4)}(q)$ has $\vert q\vert=1$ as natural boundary.

We remark that, as a by-product, we obtain that $Q_2(q)$ has a zero of
order $2$ and $R_2(q)$ has a simple zero at $q=-\exp(-\pi)$.

\smallskip
{\sc Case $\mathbf N=(2,6)$}. Here, we have
$$
\mathbf F_{(2,6)}(z)=
\sum _{n=0} ^{\infty}\frac {(6n)!} {(3n)!\,n!^3}z^n=
{}_{3}F_2
 \!\left [ \begin{matrix} {1/2,1/6,5/6}\\ {
 1,  1}\end{matrix} ; {\displaystyle 1728z}\right ]=
{}_{2}F_1
 \!\left [ \begin{matrix} {1/12,5/12}\\ {
   1}\end{matrix} ; {\displaystyle 1728z}\right ]^2,
$$
again with Clausen's formula \eqref{eq:Clausen} behind the identity between the
hypergeometric series.
It is well-known (cf.\ \cite[Sec.~2]{lianyau2}) 
that $\mathbf z_{(2,6)}(q)=1/(1728J(q))$.
By the considerations in the case where $\mathbf N=(6)$, 
we conclude that $\mathbf z_{(2,6)}(q)$
has a pole of order $3$ at all points
 $q=\exp(2i\pi\tau)$, where $\tau$ runs through the
elements of the orbit of $\rho$ under $SL_2(\mathbb Z)$. 
The point $\exp(2i\pi\rho)=-\exp(-\pi\sqrt3)$ 
is the element of smallest
modulus in this set, and the other elements become dense near
the boundary $\vert q\vert=1$. 
In particular, the Taylor
expansion of $\mathbf z_{(2,6)}(q)$ has radius of convergence 
$\exp(-\pi\sqrt 3)=\exp(-\pi\cot(\pi/6))=0.00433342\dots$,
and $\mathbf z_{(2,6)}(q)$ has $\vert q\vert=1$ as natural boundary.

\smallskip
{\sc Case $\mathbf N=(2,2,2)$}. Here, we have
$$
\mathbf F_{(2,2,2)}(z)=
\sum _{n=0} ^{\infty}\frac {(2n)!^3} {n!^6}z^n=
{}_{3}F_2
 \!\left [ \begin{matrix} {1/2,1/2,1/2}\\ {
 1,  1}\end{matrix} ; {\displaystyle 64z}\right ]=
{}_{2}F_1
 \!\left [ \begin{matrix} {1/4,1/4}\\ {
   1}\end{matrix} ; {\displaystyle 64z}\right ]^2,
$$
again with Clausen's formula \eqref{eq:Clausen} behind the identity between the
hypergeometric series.
By applying \eqref{eq:trans}, we can relate the $_2F_1$-series to the
$_2F_1$-series from the case where $\mathbf N=(2,2)$. Namely, we have
$$
{}_{2}F_1
 \!\left [ \begin{matrix} {1/4,1/4}\\ {
   1}\end{matrix} ; {\displaystyle 64z}\right ]
=
  {} _{2} F _{1} \!\left [ \begin{matrix} {1/2,1/2}\\ { 
   1}\end{matrix} ; {\displaystyle {\frac {1 - {\sqrt{1 - 64z}}} 2}}\right ],
$$
which enables us to identify $\frac {1} {2}\big({ {1 - {\sqrt{1 - 64\mathbf
        z_{(2,2,2)}(q)}}} }\big)$ with $16\mathbf z_{(2,2)}(q)$, or,
explicitly, 
\begin{equation} \label{eq:z222/z22} 
\mathbf z_{(2,2,2)}(q)=\mathbf z_{(2,2)}(q)-16\mathbf z_{(2,2)}^2(q).
\end{equation}
The expression on the right-hand side of \eqref{eq:z222/z22} can be
rewritten in the form
\begin{equation} \label{eq:z222} 
\mathbf z_{(2,2,2)}(q)=q
\prod _{n=1} ^{\infty}\frac {(1-q^{2n})^{24}} {(1-(-q)^n)^{24}}=
-\frac {H^{24}(q^2)} {H^{24}(-q)}.
\end{equation}
(This follows again by arguments as in \cite[Sections~3 and 4]{GarvAA}.)
From \eqref{eq:z222}, it is evident that the radius of convergence of
$\mathbf z_{(2,2,2)}(q)$ is $1\!=\!\exp(-\pi\cot(\pi/2))$, and
that $\mathbf z_{(2,2,2)}(q)$ has $\vert q\vert=1$ as natural boundary.
We remark that this is consistent with Lemma~\ref{lem:expinf}.$(iii)$
and Lemma~\ref{lem:qz}.$(iii)$, which imply that $\mathbf
  z_{(2,2,2)}(q)$ has an essential singularity at
  $q=-\exp(-\pi\cot(\pi/2))=-1$.  

\section{Proof of Proposition~\ref{prop:1}}
\label{sec:prop}

By Theorem~\ref{theo:2}.$(i)$, the radius of convergence of ${\bf
  q}_{\mathbf{N}} (z)$ is $1/C$. Moreover, 
by Theorem~\ref{theo:1}, the Taylor coefficients of ${\bf
  q}_{\mathbf{N}} (z)$ at $0$ are all positive. Hence,  ${\bf
  q}'_{\mathbf{N}} (z)>0$ for all $z$ in the segment $[0,1/C)$.
By the inverse function theorem, ${\bf
  q}_{\mathbf{N}} (z)$ has a compositional inverse in a domain
containing $[0,1/C)$. This implies the first assertion.

The additional assertion on the analytic nature of 
${\bf z}_{\mathbf{N}}(q)$ around ${\bf q}_{\mathbf{N}}(1/C)$
results from Lemma~\ref{lem:exp1/C}.$(ii)$, $(iii)$.
Strictly speaking, the lemma makes no assertion about the angle
(depending on $\ep$) of the slit neighbourhood at $1/C$.
However, it is not very difficult to see that the assertion that there
exists such a slit neighbourhood with an arbitrary $\ep>0$ 
can be derived by an appropriate refinement of the proof of
Lemma~\ref{lem:exp1/C}. 

The argument for the assertion on the analytic nature of 
${\bf z}_{\mathbf{N}}(q)$ around the point\break 
$-\exp\big(-\pi\cot(\pi/M)\big)$ given in the second paragraph after
Conjecture~\ref{conj:5} is
similar. It is based on the singular expansion of ${\bf q}_{\mathbf{N}}(z)$
at $z=\infty$ given in Lemma~\ref{lem:expinf} and its consequence
for the singular expansion of ${\bf z}_{\mathbf{N}}(q)$ at
$-\exp\big(-\pi\cot(\pi/M)\big)$ given in Lemma~\ref{lem:qz}.

\section{What do we need to prove Conjecture~\ref{conj:1}?}
\label{sec:What}

\begin{figure}
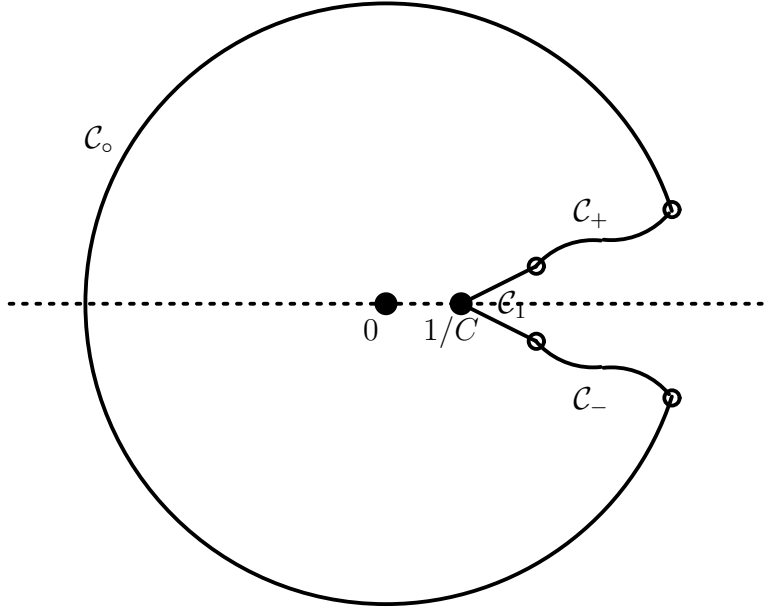

\centertexdraw{
  \drawdim truecm  \linewd.05
\larc r:4 sd:18 ed:342
\DickPunkt(0 0) 
\htext(-.3 -0.5){$0$}
\DickPunkt(1 0) 
\htext(.5 -0.6){$1/C$}
\htext(1.5 -.2){$\mathcal C_1$}
\htext(-4 2){$\mathcal C_\circ$}
\htext(2.5 1){$\mathcal C_+$}
\htext(2.5 -1.5){$\mathcal C_-$}
  \move(1 0)
  \rlvec(1 .5) 
  \move(1 0)
  \rlvec(1 -.5) 
\move(2.75 -.15)
\larc r:1 sd:84 ed:140
\move(3. 1.85)
\larc r:1 sd:264 ed:320
\move(2.75 .15)
\larc r:1 sd:220 ed:276
\move(3. -1.85)
\larc r:1 sd:40 ed:96
\move(2 .5)
\larc r:.1 sd:0 ed:360
\move(2 -.5)
\larc r:.1 sd:0 ed:360
\move(3.8 1.25)
\larc r:.1 sd:0 ed:360
\move(3.8 -1.25)
\larc r:.1 sd:0 ed:360
  \lpatt(.05 .13)
  \move(-5 0)
  \rlvec(10 0) 
\move(0 4)
}
\vskip65pt
\caption{The contour $\mathcal C$}
\label{fig:2}
\end{figure}

By the Lagrange inversion formula (cf.\ \cite[Theorem~1.9b with $R(x)=x$]{HenrAA}), we can
express the $m$-th Taylor coefficient of ${\bf z}_{\mathbf N}(q)$ in the form
\begin{equation} \label{eq:Lagr}
\frac {1} {2i\pi m}\int _{\mathcal D} ^{}\frac {\dd z} {{\bf q_{\mathbf N}}^m(z)}, 
\end{equation}
where $\mathcal D$ is a sufficiently small 
contour that encircles the origin once in
positive (that is, counter-clockwise) direction. 
We now deform $\mathcal D$ to a contour $\mathcal C$ consisting of
four parts,
$$\mathcal C=\mathcal C_1\cup \mathcal C_+\cup
\mathcal C_- \cup \mathcal C_\circ,$$
where 
\begin{enumerate} 
\item $\mathcal C_1$ is a (small) 
piece of length $\ell$, say, that passes through
$1/C$,
\item $\mathcal C_\circ$ is a segment $\{Re^{i\th}:\de\le\th\le2\pi-\de\}$ 
of a (large) circle of radius $R$, for some $\de>0$,
\item $\mathcal C_+$ is a path that connects the ends of $\mathcal
C_1$ and $\mathcal C_\circ$ 
with positive real parts,
\item $\mathcal C_+$ is a path that connects the ends of $\mathcal
C_1$ and $\mathcal C_\circ$ 
with negative real parts,
\item $\vert \mathbf q_{\mathbf N}(z)\vert\ge\mathbf
q_{\mathbf N}(1/C)$ for all $z\in \mathcal C$, 
\end{enumerate}
{\it if there exists such a contour} (see Figure~\ref{fig:2} for a symbolic
illustration of such a contour).
There is no problem with the existence of $\mathcal C_1$ since
Lem\-ma~\ref{lem:exp1/C}.$(ii)$, $(iii)$ says that
the coefficient of $1-Cz$ in the singular expansion of 
$\mathbf q_{\mathbf N}(z)$
about $z=1/C$ must be negative, nor is there with
the existence of $\mathcal C_\circ$, the latter due to
Theorem~\ref{theo:3} and Lemma~\ref{lem:exp>q(1/C)}.
The problem here is the question
of existence of a suitable path $\mathcal C_+$ connecting a point
in the neighbourhood of $\mathbf q_{\mathbf N}(1/C)$ with a point
``far out" (these points being indicated by circles in
Figure~\ref{fig:2}) 
so that, along the path, all values of $\mathbf q_{\mathbf
N}(z)$ are larger in modulus than $\mathbf q_{\mathbf N}(1/C)$. 
There cannot be any doubt that, if $\Phi\ge4$, such a path exists,
but, unfortunately, we have not been able to establish this.
It is easy to see that (by exploiting that ${\bf q}_{\bf
  N}(\overline z)=\overline{{\bf q}_{\bf
  N}(z)}$), once we know that a suitable path $C_+$
exists, there is as well a suitable path $C_-$.

Let us continue {\it under the assumption that paths $C_+$ and
$C_-$} as described above {\it exist}. 
Then we may estimate
$$
\left\vert
\frac {1} {2i\pi m}\int _{\mathcal D} ^{}\frac {\dd z} {{\bf q_{\mathbf
N}}^m(z)}
\right\vert=
\left\vert
\frac {1} {2i\pi m}\int _{\mathcal C} ^{}\frac {\dd z} {{\bf q_{\mathbf
N}}^m(z)}
\right\vert
\le 
\frac {b} {m\cdot{\bf q_{\mathbf
N}}^m(1/C)},
$$
where $b$ is some constant. Since the contour integral \eqref{eq:Lagr}
gives the coefficient of $q^m$ in the Taylor expansion of 
$\mathbf z_{\mathbf N}(q)$, this implies that the radius
of convergence of the series
$\mathbf z_{\mathbf N}(q)$ must be at least ${\bf q_{\mathbf
N}}(1/C)$. From Theorem~\ref{theo:4}.$(ii)$ it then follows that
it must be {\it exactly} ${\bf q_{\mathbf N}}(1/C)$. 

\medskip
Whether or not ${\bf q_{\mathbf N}}(1/C)$ is the only point of
singularity of ${\bf z_{\mathbf N}}(q)$ on the boundary of its disk
of convergence would have to be decided by an additional argument.

\section{Conjecture \ref{conj:2} implies the first assertion in 
Conjecture \ref{conj:1}}
\label{sec:conj}

Let $\Phi\ge4$. Conjecture~\ref{conj:2} says that all but a finite
number of coefficients in the Taylor series expansion of ${\bf z}_{\bf
  N}(q)$ are negative. Let this Taylor series expansion be
$\sum _{m=0} ^{\infty}{\sf z}_mq^m$.
A well-known theorem of Pringsheim (cf.\
\cite[Theorem~IV.6]{FlSeAA}) 
says that, if a function $f(z)$ can be represented in a neighbourhood
of $z=0$ as a power series with non-negative coefficients with radius
of convergence $R$, then $f(z)$ has a singularity at $z=R$. If we
apply this theorem to $-\sum _{m=4} ^{\infty}{\sf z}_mq^m$, 
then we conclude that the radius of convergence of
${\bf z}_{\bf  N}(q)$ is a positive real number which is at the same
time a singularity of ${\bf z}_{\bf  N}(q)$. 
By Proposition~\ref{prop:1}, we know that 
${\bf z}_{\bf  N}(q)$ is analytic at
all points $q\in [0,{\bf q}_{\bf N}(1/C))$. Thus, we know that the
radius of convergence of ${\bf z}_{\bf  N}(q)$ must be at least
${\bf q}_{\bf N}(1/C)$. On the other hand, by Theorem~\ref{theo:4} we
know that it cannot be larger, whence the conclusion.
\hfill\qed

\section{Conjecture \ref{conj:1} implies a weak version of
Conjecture \ref{conj:2}.$(ii)$ and $(iii)$}
\label{sec:conj1weak}

Here we prove that Conjecture~\ref{conj:1} implies that, if
$\Phi_{\mathbf N}\ge4$, almost all coefficients in the Taylor
expansion of ${\bf z}_{\bf N}(q)$ at $q=0$ are negative.

\begin{prop} \label{prop:conj1weak}
Let $N_1,N_2, \ldots, N_k$ be positive integers, all at least $2$, such that
$\Phi_{\mathbf{N}}\ge4$, and assume that Conjecture~{\em\ref{conj:1}}
holds. Then there exists an $A_{\mathbf N}>0$ such that 
the coefficient of $q^m$ in
the Taylor series of ${\bf z}_{\bf N}(q)$ is negative for every
$m>A_{\mathbf N}$.
\end{prop}

\begin{proof}
By Lemma~\ref{lem:exp1/C}.$(ii)$, $(iii)$, we know that $\mathbf
q_{\mathbf N}(z)$ admits a singular expansion at $z=1/C$ of the form
\begin{multline*} 
\mathbf q_{\mathbf
  N}(z)=\mathbf q_{\mathbf
  N}(1/C)+{\sf q}_1(1-Cz)+{\sf q}_2(1-Cz)^2+\cdots\\
+{\sf q}_{d-1}(1-Cz)^{d-1}+
Q_{d}(1-Cz)^{d}\widetilde L(z)
+\mathcal O\big((1-Cz)^{d}\big),
\end{multline*}
where $d=\left\lfloor\frac {\Phi-2} {2}\right\rfloor$, 
${\sf q}_1<0$, and $(-1)^{d+1}Q_{d}>0$, where 
$$\widetilde L(z)=\begin{cases} 
(1-Cz)^{\frac {1} {2}}&\text{if $\Phi$
  is odd},\\
-\log(1-Cz)&\text{if $\Phi$
  is even}.
\end{cases}$$
By Lemma~\ref{lem:qz2}, respectively Lemma~\ref{lem:qz3}, it follows
that $\mathbf
z_{\mathbf N}(q)$ admits a singular expansion at $q=\mathbf
q_{\mathbf N}(1/C)$ of the form
\begin{multline} \label{eq:zsing} 
\mathbf z_{\mathbf
  N}(q)=\frac {1} {C}
  +{\sf z}_1\big(q-\mathbf
q_{\mathbf N}(1/C)\big)+{\sf z}_2\big(q-\mathbf
q_{\mathbf N}(1/C)\big)^2+\cdots\\
+{\sf z}_{d-1}\big(q-\mathbf
q_{\mathbf N}(1/C)\big)^{d-1}+
Z_{d}\big(q-\mathbf
q_{\mathbf N}(1/C)\big)^{d}\widehat L(q)
+\mathcal O\big(\big(q-\mathbf
q_{\mathbf N}(1/C)\big)^{d+1}\big),
\end{multline}
where $d=\left\lfloor\frac {\Phi-2} {2}\right\rfloor$, 
${\sf z}_1=-1/(C{\sf q}_1)$, and $Z_{d}>0$, where 
$$\widehat L(z)=\begin{cases} 
\big(\mathbf
q_{\mathbf N}(1/C)-q\big)^{\frac {1} {2}}&\text{if $\Phi$
  is odd},\\
-\log\big(\mathbf
q_{\mathbf N}(1/C)-q\big)&\text{if $\Phi$
  is even}.
\end{cases}$$

If we now use Conjecture~\ref{conj:1}, then
the standard theorems of singularity analysis (see
\cite[Ch.~VI]{FlSeAA}) imply that the term in \eqref{eq:zsing}
containing $\widehat L(q)$ contributes the main term to the
asymptotics of the Taylor coefficients of $\mathbf z_{\mathbf
  N}(q)$. More precisely, for $m\to\infty$, the $m$-th Taylor
coefficient of $\mathbf z_{\mathbf
  N}(q)$ is equal to
\begin{equation} \label{eq:zasymp}
-const.\frac {1+o(1)} {\mathbf
q^m_{\mathbf N}(1/C)\,m^{\frac {\Phi} {2}}}, 
\end{equation}
where $const.$ is a positive constant which can be computed explicitly
in terms of $Z_d$, $\Phi$, and $m$. The claim of the proposition is
now obvious. 
\end{proof}

\section*{Acknowledgements}
The authors thank Bruce Berndt and Heng-Huat Chan for very
helpful correspondence concerning the discussion of modular mirror
maps presented here in Section~\ref{sec:Phi}.

\def\refname{Bibliography}

\end{document}